\documentclass[oneside,english,american]{amsart}
\usepackage[T1]{fontenc}
\usepackage[latin9]{inputenc}
\usepackage{textcomp}
\usepackage{amsthm}
\usepackage{amstext}
\usepackage{amssymb}
\usepackage{esint}

\makeatletter

\newcommand{\mathcircumflex}[0]{\mbox{\^{}}}

\numberwithin{equation}{section}
\numberwithin{figure}{section}
\theoremstyle{plain}
\newtheorem{thm}{\protect\theoremname}
  \theoremstyle{plain}
  \newtheorem{lem}[thm]{\protect\lemmaname}
  \theoremstyle{plain}
  \newtheorem{cor}[thm]{\protect\corollaryname}

\usepackage{babel}

\makeatother

\usepackage{babel}
  \addto\captionsamerican{\renewcommand{\corollaryname}{Corollary}}
  \addto\captionsamerican{\renewcommand{\lemmaname}{Lemma}}
  \addto\captionsamerican{\renewcommand{\theoremname}{Theorem}}
  \addto\captionsenglish{\renewcommand{\corollaryname}{Corollary}}
  \addto\captionsenglish{\renewcommand{\lemmaname}{Lemma}}
  \addto\captionsenglish{\renewcommand{\theoremname}{Theorem}}
  \providecommand{\corollaryname}{Corollary}
  \providecommand{\lemmaname}{Lemma}
\providecommand{\theoremname}{Theorem}

\begin{document}

\title{Gaussian Free Fields and KPZ Relation in $\mathbf{\mathbb{R}^{4}}$}

\author{Linan Chen and \foreignlanguage{english}{Dmitry Jakobson}}

\selectlanguage{english}%

\address{Department of Mathematics and Statistics, McGill University, 805
Sherbrooke Street West, Montréal, QC H3A 0B9, Ca\-na\-da.}

\email{lnchen@math.mcgill.ca, jakobson@math.mcgill.ca}

\thanks{L.C. was partly supported by FQRNT. D.J was partly supported by NSERC,
FQRNT and Dawson fellowships.}



\date{\today}
\selectlanguage{american}%
\begin{abstract}
This work aims to extend part of the two dimensional results of Duplantier
and Sheffield on Liouville quantum gravity \cite{DS1} to four dimensions,
and indicate possible extensions to other even-dimensional spaces
$\mathbb{R}^{2n}$ as well as Riemannian manifolds. 

Let $\Theta$ be the Gaussian free field on $\mathbb{R}^{4}$ with
the underlying Hilbert space $H^{2}\left(\mathbb{R}^{4}\right)$ and
the inner product $\left(\left(I-\Delta\right)^{2}\cdot,\cdot\right)_{L^{2}}$,
and $\theta$ a generic element from $\Theta$. We consider a sequence
of random Borel measures on $\mathbb{R}^{4}$, denoted by $\left\{ m_{\epsilon_{n}}^{\theta}\left(dx\right):n\geq1\right\} $,
each of which is absolutely continuous with respect to the Lebesgue
measure $dx$, and the density function is given by the exponential
of a centered Gaussian family parametrized by $x\in\mathbb{R}^{4}$.
We show that with probability 1, $m_{\epsilon_{n}}^{\theta}\left(dx\right)$
weakly converges as $\epsilon_{n}\downarrow0$, and the limit measure
can be ``formally'' written as ``$m^{\theta}\left(dx\right)=e^{2\gamma\theta\left(x\right)}dx$''.
In this setting, we also prove a KPZ relation, which is the quadratic
relation between the scaling exponent of a bounded Borel set on $\mathbb{R}^{4}$
under the Lebesgue measure and its counterpart under the random measure
$m^{\theta}\left(dx\right)$. 

Our approach is similar to the one used in \cite{DS1} with adaptations
to $\mathbb{R}^{4}$. 
\end{abstract}
\maketitle

\section{Introduction }

\selectlanguage{english}%
Random measures have long been considered in $2$-dimensional conformal
field theory and quantum gravity since the work of Knizhnik, Polyakov
and Zamolodchikov \cite{KPZ}. Recently, a probabilistic proof of
the formula due to Knizhnik, Polyakov and Zamolodchikov was given
by Duplantier and Sheffield in \cite{DS1}. On the unit planar disc
$\mathbb{D}$, Duplantier and Sheffield construct the Liouville quantum
gravity measure ``$e^{\gamma h\left(z\right)}dz$'', where $dz$
is the Lebesgue measure on $\mathbb{D},$ $\gamma$ is a properly
chosen positive constant and $h$ is an instance of the Gaussian free
field (GFF) on $\mathbb{D}$ with the Dirichlet inner product. To
be specific, they prove that the random measure exists as the weak
convergence limit of $\epsilon^{\gamma^{2}/2}e^{\gamma h_{\epsilon}\left(z\right)}dz$
as $\epsilon\downarrow0$, where $h_{\epsilon}\left(z\right)$ is
the circular average of $h$ over the circle centered at $z$ with
radius $\epsilon.$ They further show that there is a quadratic relation,
known as the KPZ relation, between the scaling exponent of a random
set under the Lebesgue measure, and its counterpart under the quantum
gravity measure. Another derivation was given in \cite{DB}. Gaussian
free field in dimension $2$ has also been considered in \cite{HMP}
and numerous other papers.

\selectlanguage{american}%
In this article, we generalize part of the results from \cite{DS1}
to four dimensions. We define the Euclidean GFF on $\mathbb{R}^{4}$,
denoted by $\Theta$, with the inner product determined by the Bessel
operator $\left(I-\Delta\right)^{2}$. In other words, the underlying
Hilbert space of $\Theta$ is given by the Sobolev space $H^{2}\left(\mathbb{R}^{4}\right)$
with the inner product $\left(\left(I-\Delta\right)^{2}\cdot,\cdot\right)_{L^{2}}$.
In this setting, we prove (Section 2, Theorem \ref{thm: construction of the measure})
that given $0<\gamma^{2}<2\pi^{2}$, almost every $\theta\in\Theta$
admits a random measure on $\mathbb{R}^{4}$ which ``formally''
has the density $e^{2\gamma\theta\left(x\right)}$ with respect to
the Lebesgue measure $dx$ on $\mathbb{R}^{4}$. We also show that
this random measure satisfies a KPZ relation similar to the one in
the two dimensional case. Namely, if $\kappa\in\left[0,1\right]$
is the scaling exponent of a bounded Borel set in $\mathbb{R}^{4}$
under the Lebesgue measure, and $K\in\left[0,1\right]$ is the scaling
exponent of the same set but under the random measure obtained above
(both $\kappa$ and $K$ will be defined in Section 4), then $\kappa$
and $K$ satisfy the following quadratic relation (Section 4, Theorem
\ref{thm:KPZ}): 
\[
\kappa=K\left(1-\frac{\gamma^{2}}{16\pi^{2}}\right)+\frac{\gamma^{2}}{16\pi^{2}}K^{2}.
\]

Our proof follows the outline of the proof in \cite{DS1} with adaptations
to four dimensions. Mainly we have to overcome (both in ``designing''
the model to work with and in technical details) the difficulties
caused by the absence in our problem the two dimensional conformal
structure. To interpret rigorously an instance $\theta$ of the GFF
on the entire Euclidean space $\mathbb{R}^{4}$, we adopt the theory
of the abstract Wiener space. A key ingredient in this theory is the
underlying Hilbert space whose inner product determines the covariance
structure of the field. It is already known that in order to obtain
a measure which ``formally'' has the exponential of $\theta\left(x\right)$
as the density with respect to $dx$, the covariance function $\textrm{Cov}\left(\theta\left(x\right),\theta\left(y\right)\right)$
can at most grow at the rate of $-\log\left|x-y\right|$ when $\left|x-y\right|$
is small. Taking this into account, $H^{2}\left(\mathbb{R}^{4}\right)$
with the inner product $\left(\left(I-\Delta\right)^{2}\cdot,\cdot\right)_{L^{2}}$
becomes our natural choice. Also this way of defining the GFF makes
it possible, in certain situations, to obtain explicit formulas of
the covariance function. To construct the random measure and thereafter
to study it, we always need to relate it to a sequence of approximating
measures which converges in some proper sense. So it is our intention
to choose the approximating measures appropriately so they will be
convenient to work with. In the two dimensional case in \cite{DS1},
the approximating measures are in terms of the circular averages of
the GFF on $\mathbb{D}$. In fact, the properties of the Gaussian
family consisting of these circular averages play an important role
in the proof. For example, if $h$ is an instance of the GFF on $\mathbb{D}$,
then given any $z\in\mathbb{D}$, the one-parameter family $\left\{ h_{\epsilon}\left(z\right):0<\epsilon\leq1\right\} $
has up to a time change the same distribution as a standard Brownian
motion. Such properties are derived from the Green's function of the
Laplace operator $\Delta$ on $\mathbb{D}$, which, in particular,
is harmonic. Therefore, it should not be surprising that the trivial
analogue in four dimensions, that is, the family of spherical averages
of $\theta$, fails to have such properties, which makes it a less
than optimal substitute for $h_{\epsilon}$ in carrying out this project
on $\mathbb{R}^{4}$. In Section 2, we present one possible replacement
for $h_{\epsilon}$ in four dimensions which still has simple and
concrete geometric interpretations (in fact, it is given by a functional
of the spherical average of $\theta$), but possesses, to a large
extent, similar properties to those of $h_{\epsilon}$ in two dimensions.
In Section 3, we use the results from Section 2 to build the approximating
measures, and then prove they almost surely admit a limit measure
in the sense of weak convergence. In Section 4, we lay out an outline
to derive the KPZ relation and the proofs of the main results are
collected in Section 5. Other work on the KPZ relation in higher dimensions
with different settings can be found in \cite{JJRV,RV}.

Our original interest in constructing such a random measure lies in
its potential applications in the study of conformal classes of Riemannian
metrics. In fact, another more geometric point of view on the GFF
on a planar domain or more generally on a surface $\Sigma$, is to
consider it as a measure on a conformal class of metrics on $\Sigma$,
where the measure is constructed with the help of a reference metric
$g_{0}$ on $\Sigma$, but where the result does not depend on $g_{0}$.
It seems natural to generalize this approach to conformal classes
of metrics on higher-dimensional manifolds. It turns out that on a
compact four-dimensional manifold $\mathfrak{M}$, a natural replacement
for the Laplace-Beltrami operator $\Delta$ (that is used in the construction
of the GFF on surfaces) is the $4$-th order \emph{Paneitz operator}
that arises in the conformal geometry. More generally, on compact
$2n$-dimensional manifolds it seems natural to use the \emph{dimension-critical
GJMS operator} in the construction of higher-dimensional analogues
of the two-dimensional GFF. This will be further explained in the
second part of Section 6.\\

\noindent \textbf{Acknowledgments.} The authors thank L. Addario-Berry
and D. Stroock for their valuable comments on an earlier version of
this article. We also thank R. Adler, I. Binder, Y. Canzani, B. Duplantier,
S. Klevtsov, R. Ponge, P. Sarnak, S. Sheffield, J. Taylor, I. Wigman,
P. Yang, S. Zelditch for many fruitful conversations.

\section{Spherical Averages of GFF on $\mathbb{R}^{4}$}

We start with a brief review of some fundamental facts about the abstract
Wiener space theory (\cite{aws} or \cite{awsrevisited}). An abstract
Wiener space is commonly used in constructions of infinite dimensional
Gaussian measures. The basic setting of an abstract Wiener space is
as follows. Given a (infinite dimensional) Banach space $\Theta$
and a (infinite dimensional) Hilbert space $H$, assume both $\Theta$
and $H$ are separable, and $H$ can be continuously embedded into
$\Theta$ as a dense subspace. Therefore if $x^{*}$ is a bounded
linear functional on $\Theta$ (denoted by $x^{*}\in\Theta^{*}$),
then there is unique $h_{x^{*}}\in H$ such that for every $h\in H$,
$\left(h,h_{x^{*}}\right)_{H}=\left\langle h,x^{*}\right\rangle $,
where $\left\langle \cdot,*\right\rangle $ refers to the action of
$\Theta^{*}$ on $\Theta$ (or more specifically in later discussions,
the action of tempered distributions on test functions). Let $\mathcal{W}$
be a probability measure on $\left(\Theta,\mathfrak{B}_{\Theta}\right)$
where $\mathfrak{B}_{\Theta}$ is the Borel $\sigma-$algebra of $\Theta$.
If $\mathcal{W}$ satisfies
\[
\mathbb{E}^{\mathcal{W}}\left[\exp\left(i\left\langle \cdot,x^{*}\right\rangle \right)\right]=\exp\left(-\frac{\left\Vert h_{x^{*}}\right\Vert _{H}^{2}}{2}\right)\mbox{ for all }x^{*}\in\Theta^{*},
\]
then the triple $\left(H,\Theta,\mathcal{W}\right)$ is called an
\emph{abstract Wiener space}. It is known (\cite{probability}, §8.3)
that given any separable Hilbert space, one can always find $\Theta$
and $\mathcal{W}$ such that $\left(H,\Theta,\mathcal{W}\right)$
forms an abstract Wiener space. Moreover, since $\left\{ h_{x^{*}}:x^{*}\in\Theta^{*}\right\} $
is also dense in $H$, the linear mapping 
\[
\mathcal{I}:\; h_{x^{*}}\in H\mapsto\mathcal{I}\left(h_{x^{*}}\right)\equiv\left\langle \cdot,x^{*}\right\rangle \in L^{2}\left(\mathcal{W}\right)
\]
can be uniquely extended as a linear isometry from $H$ to $L^{2}\left(\mathcal{W}\right)$.
Its images $\left\{ \mathcal{I}\left(h\right):h\in H\right\} $, known
as the \emph{Paley-Wiener integrals}, form a centered Gaussian family
whose covariance is given by 
\[
\mathbb{E}^{\mathcal{W}}\left[\mathcal{I}\left(h_{1}\right)\mathcal{I}\left(h_{2}\right)\right]=\left(h_{1},h_{2}\right)_{H}\mbox{ for all }h_{1},h_{2}\in H.
\]
We point out that although the Hilbert structure of $H$ plays an
essential role, $\mathcal{W}\left(H\right)=0$ and the choice of $\Theta$
is not unique. 

As we have mentioned in the previous section, we consider in our project
the infinite dimensional Gaussian measure on the space of certain
tempered distributions on $\mathbb{R}^{4}$, with the underlying Hilbert
space given by the Sobolev space $H\equiv H^{2}\left(\mathbb{R}^{4}\right)$,
which is the completion of the real valued Schwartz test function
space $\mathcal{S}\left(\mathbb{R}^{4}\right)$ under the inner product
\[
\left(f_{1},f_{2}\right)_{H}\equiv\int_{\mathbb{R}^{4}}\left(I-\Delta\right)^{2}f_{1}\left(x\right)f_{2}\left(x\right)dx\mbox{ for all }f_{1},f_{2}\in\mathcal{S}\left(\mathbb{R}^{4}\right).
\]
Then, given this particular choice of $H$, our notion of the \emph{Gaussian
free field} on $\mathbb{R}^{4}$ refers to any probability space $\left(\Theta,\mathfrak{B}_{\Theta},\mathcal{W}\right)$
such that $\left(\Theta,H,\mathcal{W}\right)$ forms an abstract Wiener
space. For example, if $\tilde{\Theta}$ is the space of continuous
functions $\theta:\,\mathbb{R}^{4}\rightarrow\mathbb{R}$ satisfying
\[
\lim_{\left|x\right|\rightarrow\infty}\left(\log\left(e+\left|x\right|\right)\right)^{-1}\left|\theta\left(x\right)\right|=0,
\]
then $\Theta$ can be chosen as the image of $\tilde{\Theta}$ under
the Bessel operator $\left(I-\Delta\right)^{\frac{1}{4}}$, i.e.,
\[
\Theta=\left\{ \left(I-\Delta\right)^{\frac{1}{4}}\theta:\,\theta\in\tilde{\Theta}\right\} .
\]
From this we observe that $\Theta$ consists of tempered distributions
which in general are not defined point-wise. Nonetheless, we can understand
some properties of the GFF by studying the Paley-Wiener integrals,
which can be viewed as ``generalized'' action of certain tempered
distributions on $\Theta$.

In addition, if $H^{-2}=H^{-2}\left(\mathbb{R}^{4}\right)$ is the
Hilbert space consisting of tempered distributions $\mu$ such that
\[
\left\Vert \mu\right\Vert _{H^{-2}}^{2}\equiv\frac{1}{\left(2\pi\right)^{4}}\int_{\mathbb{R}^{4}}\left(1+\left|\xi\right|^{2}\right)^{-2}\left|\hat{\mu}\left(\xi\right)\right|^{2}d\xi<\infty
\]
where $\hat{\mu}$ is the Fourier transform (without the factor $\left(2\pi\right)^{-2}$
in the definition) of $\mu$, then we can identify $H$ with $H^{-2}$
since $\left(I-\Delta\right)^{-2}:\, H^{-2}\rightarrow H$ is obviously
a linear isometry. We will abuse the notation%
\footnote{The subscript of ``$h_{\nu}$'' is an element of $H^{-2}$, not
to be confused with ``$h_{x^{*}}$'' in the definition of the abstract
Wiener space where $x^{*}\in\Theta^{*}$. %
} by denoting ``$h_{\nu}$'' the image of $\nu\in H^{-2}$ under
$\left(I-\Delta\right)^{-2}$. Then $h_{\nu}$ is the unique element
in $H$ such that $\left\langle h,\nu\right\rangle =\left(h,h_{\nu}\right)_{H}$
for all $h\in H$, which suggests that the corresponding Paley-Wiener
integral $\mathcal{I}\left(h_{\nu}\right)$ can be viewed as a ``representation''
of the action of $\nu$ on $\Theta$, even though $\nu$ is not in
$\Theta^{*}$ and $\mathcal{I}\left(h_{\nu}\right)\left(\theta\right)$
is only defined for almost every $\theta\in\Theta$. Meanwhile, $\left\{ \mathcal{I}\left(h_{\nu}\right):\nu\in H^{-2}\right\} $
is also a Gaussian family whose covariance is given by 
\[
\mathbb{E}^{\mathcal{W}}\left[\mathcal{I}\left(h_{\nu_{1}}\right)\mathcal{I}\left(h_{\nu_{2}}\right)\right]=\left(h_{\nu_{1}},h_{\nu_{2}}\right)_{H}=\left(\nu_{1},\nu_{2}\right)_{H^{-2}}.
\]

With these in mind, as a natural analogue of the 2D circular average,
we consider the spherical average of the GFF on $\mathbb{R}^{4}$.
To this end, for every $x\in\mathbb{R}^{4}$ and $\epsilon>0$, denote
$\sigma_{\epsilon}^{x}$ the tempered distribution determined by
\[
\left\langle f,\sigma_{\epsilon}^{x}\right\rangle \equiv\frac{1}{2\pi^{2}\epsilon^{3}}\int_{S_{\epsilon}\left(x\right)}f\left(y\right)d\sigma\left(y\right)\mbox{ for all }f\in\mathcal{S}\left(\mathbb{R}^{4}\right),
\]
where $S_{\epsilon}\left(x\right)$ is the sphere centered at $x$
with radius $\epsilon$, and $d\sigma$ is the surface area measure
on $S_{\epsilon}\left(x\right)$. Clearly, the action of $\sigma_{\epsilon}^{x}$
is to take the spherical average of $f$ over $S_{\epsilon}\left(x\right)$.
It is an easy matter to verify that $\sigma_{\epsilon}^{x}\in H^{-2}$.
In fact, one only needs to write down the Fourier transform of $\sigma_{\epsilon}^{x}$
as 
\begin{equation}
\hat{\sigma_{\epsilon}^{x}}\left(\xi\right)=2\left(\epsilon\left|\xi\right|\right)^{-1}J_{1}\left(\epsilon\left|\xi\right|\right)e^{i\left(x,\xi\right)_{\mathbb{R}^{4}}}\label{eq:fourier transf. of average}
\end{equation}
where $J_{k}\left(r\right)$ is the Bessel function of order $k\in\mathbb{N}$,
and use the fact that $J{}_{k}\left(r\right)$ is asymptotic to $r^{-1/2}$
when $r$ is large.

As we have indicated in the introduction (and as we will confirm in
the next lemma), the spherical average of the GFF on $\mathbb{R}^{4}$
does not behave as ``nicely'' as the circular average of the GFF
in two dimensions. For one thing $\left\{ \mathcal{I}\left(h_{\sigma_{\epsilon}^{x}}\right):\epsilon>0\right\} $
fails to be a reversed Markov process. An intuitive way to view this
is that, the spherical average does not bear enough information in
itself for this Gaussian family parametrized by radius $\epsilon>0$
to be (reversed) Markovian. It might also be helpful to relate this
to the following analogous problem: when solving PDEs with higher
order differential operator on a domain with boundary, one often needs
more than one boundary condition (e.g., both the Dirichlet and the
Neumann boundary conditions) to uniquely determine the solution. Inspired
by this idea, besides the average itself we will also ``collect''
one more piece of information about the GFF from each sphere, which
is the ``derivative'' of the average with respect to the radius.
Namely, for every $x\in\mathbb{R}^{4}$ and $\epsilon>0$, denote
$d\sigma_{\epsilon}^{x}$ the tempered distribution given by $\left\langle f,d\sigma_{\epsilon}^{x}\right\rangle \equiv\frac{d}{d\epsilon}\left\langle f,\sigma_{\epsilon}^{x}\right\rangle $
for all $f\in\mathcal{S}\left(\mathbb{R}^{4}\right)$, then the action
of $d\sigma_{\epsilon}^{x}$ can be viewed as to take the derivative
of the spherical average of the GFF in the radial direction. It follows
trivially from (\ref{eq:fourier transf. of average}) that
\begin{equation}
\hat{d\sigma_{\epsilon}^{x}}\left(\xi\right)=\frac{d}{d\epsilon}\hat{\sigma_{\epsilon}^{x}}\left(\xi\right)=-2\epsilon^{-1}J_{2}\left(\epsilon\left|\xi\right|\right)e^{i\left(x,\xi\right)_{\mathbb{R}^{4}}}.\label{eq:fourier transf. for derivative of average}
\end{equation}
In particular, $d\sigma_{\epsilon}^{x}$ is also in $H^{-2}$ and
so $\left\{ \mathcal{I}\left(h_{\sigma_{\epsilon}^{x}}\right),\mathcal{I}\left(h_{d\sigma_{\epsilon}^{x}}\right):x\in\mathbb{R}^{4},\epsilon>0\right\} $
forms a centered Gaussian family whose covariance is determined by
the $H^{-2}$ inner product of $\left\{ \sigma_{\epsilon}^{x},d\sigma_{\epsilon}^{x}:x\in\mathbb{R}^{4},\epsilon>0\right\} $.

The next lemma in some sense validates our decision to take $d\sigma_{\epsilon}^{x}$
into account. It shows that by putting $\mathcal{I}\left(h_{\sigma_{\epsilon}^{x}}\right)$
and its ``derivative'' $\mathcal{I}\left(h_{d\sigma_{\epsilon}^{x}}\right)$
together%
\footnote{This idea came from discussions with Daniel W. Stroock when the first
author was studying at MIT.%
}, not only does the Gaussian family recover the reversed Markov property
in the concentric case (with $x\in\mathbb{R}^{4}$ fixed, parametrized
by $\epsilon>0$ only), the non-concentric family (parametrized by
both $\epsilon>0$ and $x\in\mathbb{R}^{4}$) also resembles, to a
large extent, its counterpart in two dimensions. To be precise, we
define the vector-valued Gaussian random variable: 
\[
V_{\epsilon}^{x}\equiv\left(\begin{array}{c}
\mathcal{I}\left(h_{\sigma_{\epsilon}^{x}}\right)\\
\mathcal{I}\left(h_{d\sigma_{\epsilon}^{x}}\right)
\end{array}\right)\mbox{ for every }x\in\mathbb{R}^{4}\mbox{ and }\epsilon>0.
\]
Then, under certain circumstances, the covariance matrix of the Gaussian
family $\left\{ V_{\epsilon}^{x}:x\in\mathbb{R}^{4},\epsilon>0\right\} $
can be evaluated explicitly as follows. 
\begin{lem}
\label{lem:on V_t(x)}For $r\in\left(0,\infty\right)$, define the
following four matrices: 
\[
\begin{split}\mathbf{A}\left(r\right)\equiv\left(\begin{array}{cc}
K_{1}^{\prime}\left(r\right) & K_{1}\left(r\right)/r\\
K_{1}^{\prime\prime}\left(r\right) & -K_{2}\left(r\right)/r
\end{array}\right),\, & \mathbf{B}\left(r\right)\equiv\left(\begin{array}{cc}
I_{1}\left(r\right)/r & I_{1}^{\prime}\left(r\right)\\
I_{2}\left(r\right)/r & I_{1}^{\prime\prime}\left(r\right)
\end{array}\right),\\
\mathbf{C}\left(r\right)\equiv\left(\begin{array}{cc}
I_{1}\left(r\right)/r & 0\\
I_{2}\left(r\right) & I_{1}\left(r\right)/r
\end{array}\right),\;\; & \mathbf{D}\left(r\right)\equiv\left(\begin{array}{cc}
-K_{2}\left(r\right) & K_{1}\left(r\right)/r\\
K_{1}\left(r\right)/r & 0
\end{array}\right),
\end{split}
\]
 where $I_{k},K_{k}$ are the modified Bessel functions of order $k\in\mathbb{N}$.
Then,\\
\\
(1), given $x\in\mathbb{R}^{4}$ and $\epsilon_{1}\geq\epsilon_{2}>0$,
\begin{equation}
\mathbb{E}^{\mathcal{W}}\left[V_{\epsilon_{1}}^{x}\left(V_{\epsilon_{2}}^{x}\right)^{\top}\right]=\left(-\frac{1}{4\pi^{2}}\right)\mathbf{A}\left(\epsilon_{1}\right)\mathbf{B}^{\top}\left(\epsilon_{2}\right).\label{eq:vector covariance concentric}
\end{equation}
In particular, $\left\{ V_{\epsilon}^{x}:\epsilon>0\right\} $ is
a vector-valued Gaussian reversed Markov process in the sense that
for every Borel $A\subseteq\mathbb{R}^{2}$, 
\[
\mathcal{W}\left(V_{\epsilon_{2}}^{x}\in A|\sigma\left\{ V_{\eta}^{x}:\,\eta\geq\epsilon_{1}\right\} \right)=\mathcal{W}\left(V_{\epsilon_{2}}^{x}\in A|V_{\epsilon_{1}}^{x}\right),
\]
where $\sigma\left\{ V_{\eta}^{x}:\eta\geq\epsilon_{1}\right\} $
is the $\sigma-$algebra generated by $\left\{ V_{\eta}^{x}:\eta\geq\epsilon_{1}\right\} $.
\\
\\
(2), given $x,y\in\mathbb{R}^{4}$, $x\neq y$, and $\epsilon_{1},\epsilon_{2}>0$
with $\epsilon_{1}>\left|x-y\right|+\epsilon_{2}$, 
\begin{equation}
\mathbb{E}^{\mathcal{W}}\left[V_{\epsilon_{1}}^{x}\left(V_{\epsilon_{2}}^{y}\right)^{\top}\right]=\left(-\frac{1}{2\pi^{2}}\right)\mathbf{A}\left(\epsilon_{1}\right)\mathbf{C}\left(\left|x-y\right|\right)\mathbf{B}^{\top}\left(\epsilon_{2}\right).\label{eq:vector covariance inclusion}
\end{equation}
\\
(3), given $x,y\in\mathbb{R}^{4}$, $x\neq y$, and $\epsilon_{1},\epsilon_{2}>0$
with $\left|x-y\right|>\epsilon_{1}+\epsilon_{2}$, 
\begin{equation}
\mathbb{E}^{\mathcal{W}}\left[V_{\epsilon_{1}}^{x}\left(V_{\epsilon_{2}}^{y}\right)^{\mbox{\ensuremath{\top}}}\right]=\left(-\frac{1}{2\pi^{2}}\right)\mathbf{B}\left(\epsilon_{1}\right)\mathbf{D}\left(\left|x-y\right|\right)\mathbf{B}^{\top}\left(\epsilon_{2}\right).\label{eq:vector covariance nonoverlap}
\end{equation}
 
\end{lem}
The proof of (\ref{eq:vector covariance concentric})-(\ref{eq:vector covariance nonoverlap})
relies heavily on the integral formulas and identities of Bessel functions.
The complete detailed computations are given in the appendix. Here
we only make the following observations.

First, we claim that the distribution of $\left\{ V_{\epsilon}^{x}:x\in\mathbb{R}^{4},\epsilon>0\right\} $
is invariant under isometries in spatial variables in the sense that
$\left\{ V_{\epsilon}^{\mathcal{T}\left(x\right)}:x\in\mathbb{R}^{4},\epsilon>0\right\} $
has exactly the same distribution as $\left\{ V_{\epsilon}^{x}:x\in\mathbb{R}^{4},\epsilon>0\right\} $
for every $\mathcal{T}:\mathbb{R}^{4}\rightarrow\mathbb{R}^{4}$ satisfying
$\left|\mathcal{T}\left(x\right)-\mathcal{T}\left(y\right)\right|=\left|x-y\right|$
for all $x,y\in\mathbb{R}^{4}$. Perhaps the most straightforward
way to see this is to write down the covariance matrix of the family,
or equivalently, the $H^{-2}$ inner product of $\left\{ \sigma_{\epsilon}^{x},d\sigma_{\epsilon}^{x}:x\in\mathbb{R}^{4},\epsilon>0\right\} $
in the integral form using (\ref{eq:fourier transf. of average})
and (\ref{eq:fourier transf. for derivative of average}). We will
actually do this in the appendix (formulas (\ref{eq:var sphe ave integral form})-(\ref{eq:cov sphe ave deriv integral form})).
The result shows that the only dependence of the covariance matrix
on spatial variables is through the distance between centers of the
spheres that are involved.

Second, we point out that all the matrices above $\mathbf{A}\left(r\right)$,
$\mathbf{B}\left(r\right)$, $\mathbf{C}\left(r\right)$ and $\mathbf{D}\left(r\right)$
are invertible for all $r>0$. This fact can certainly be verified
by direct computations using the explicit formulas given above, but
it also follows, more generally, from the simple fact that $d\sigma_{\epsilon}^{x}$
is linearly independent of $\sigma_{\epsilon}^{y}$ for every $x,y\in\mathbb{R}^{4}$
and $\epsilon>0$. Therefore, assuming (\ref{eq:vector covariance concentric})
is true, then given $x$ fixed and $\epsilon_{1}\geq\epsilon_{2}>0$,
the conditional expectation of $V_{\epsilon_{2}}^{x}$ conditioning
on $V_{\epsilon_{1}}^{x}$ equals 
\[
\mathbb{E}^{\mathcal{W}}\left[V_{\epsilon_{2}}^{x}\left(V_{\epsilon_{1}}^{x}\right)^{\top}\right]\left(\mathbb{E}^{\mathcal{W}}\left[V_{\epsilon_{1}}^{x}\left(V_{\epsilon_{1}}^{x}\right)^{\top}\right]\right)^{-1}V_{\epsilon_{1}}^{x}=\mathbf{B}\left(\epsilon_{2}\right)\mathbf{B}^{-1}\left(\epsilon_{1}\right)V_{\epsilon_{1}}^{x}.
\]
On the other hand, we observe that for all $\eta\geq\epsilon_{1}$,
\[
\mathbb{E}^{\mathcal{W}}\left[\left(V_{\epsilon_{2}}^{x}-\mathbf{B}\left(\epsilon_{2}\right)\mathbf{B}^{-1}\left(\epsilon_{1}\right)V_{\epsilon_{1}}^{x}\right)\left(V_{\eta}^{x}\right)^{\top}\right]=0.
\]
This means, 
\[
V_{\epsilon_{2}}^{x}-\mathbf{B}\left(\epsilon_{2}\right)\mathbf{B}^{-1}\left(\epsilon_{1}\right)V_{\epsilon_{1}}^{x}
\]
is independent of $V_{\eta}^{x}$ for all $\eta\geq\epsilon_{1}$,
which certainly implies the reversed Markov property.

Next, we observe that under the circumstances as prescribed in Lemma
\ref{lem:on V_t(x)}, the covariance matrix of $\left\{ V_{\epsilon}^{x}:x\in\mathbb{R}^{4},\epsilon>0\right\} $
is ``separable'' in the sense that it splits into factors each of
which only depends on one of the variables $\epsilon_{1}$, $\epsilon_{2}$
and $\left|x-y\right|$. A second look at the formulas (\ref{eq:vector covariance concentric})-(\ref{eq:vector covariance nonoverlap})
suggests that we should ``normalize'' $V_{\epsilon}^{x}$ by $\mathbf{B}^{-1}\left(\epsilon\right)$.
Namely, if denote $U_{\epsilon}^{x}\equiv\mathbf{B}^{-1}\left(\epsilon\right)V_{\epsilon}^{x}$,
then the previous observations imply that given $x$ fixed, $\left\{ U_{\epsilon}^{x}:\epsilon>0\right\} $
is a vector-valued Gaussian process with independent (reversed) increment
whose distribution does not depend on $x$. Moreover, (\ref{eq:vector covariance inclusion})
and (\ref{eq:vector covariance nonoverlap}) show that $\mathbb{E}^{\mathcal{W}}\left[U_{\epsilon_{1}}^{x}\left(U_{\epsilon_{2}}^{y}\right)^{\top}\right]$
only depends on $\epsilon_{1}$ and $\left|x-y\right|$ when%
\footnote{The notation ``$B_{r}\left(x\right)$'' (``$\overline{B_{r}\left(x\right)}$'')
denotes the open (closed) ball centered at $x$ with radius $r$.%
} $\overline{B_{\epsilon_{2}}\left(y\right)}\subseteq B_{\epsilon_{1}}\left(x\right)$,
and the same matrix only depends on $\left|x-y\right|$ when $\overline{B_{\epsilon_{2}}\left(y\right)}\cap\overline{B_{\epsilon_{1}}\left(x\right)}=\emptyset$.
Because $U_{\epsilon}^{x}$ has these properties, we are one step
closer to finding a plausible replacement for the circular average
of the two dimensional GFF.

Clearly, for any constant $\zeta=\left(\zeta_{1},\zeta_{2}\right)^{\top}\in\mathbb{R}^{2}$,
$\left(U_{\epsilon}^{x},\zeta\right)_{\mathbb{R}^{2}}$ is a scalar
valued Gaussian random variable (in fact, it is a Paley-Wiener integral),
which, when parametrized by $x\in\mathbb{R}^{4}$ and $\epsilon>0$,
forms a Gaussian family that preserves the properties described above.
Our goal is to find a proper $\zeta\in\mathbb{R}^{2}$ such that the
random variable 
\[
\theta\in\Theta\mapsto\left(U_{\epsilon}^{x},\zeta\right)_{\mathbb{R}^{2}}\left(\theta\right)=\zeta^{\top}\mathbf{B}^{-1}\left(\epsilon\right)V_{\epsilon}^{x}\left(\theta\right)
\]
becomes a ``legitimate'' approximation for a multiple of the value
of $\theta$ at point $x$ for every $x\in\mathbb{R}^{4}$. Namely,
we want to choose $\zeta$ so that if $\mu_{\epsilon}^{x}\in H^{-2}$
is given by 
\begin{equation}
\mu_{\epsilon}^{x}\equiv\zeta^{\top}\mathbf{B}^{-1}\left(\epsilon\right)\left(\begin{array}{c}
\sigma_{\epsilon}^{x}\\
d\sigma_{\epsilon}^{x}
\end{array}\right),\label{eq:mu_epsilon def}
\end{equation}
then $\mu_{\epsilon}^{x}$ converges to a constant multiple of the
point mass $\delta_{x}$ at $x$ as $\epsilon\downarrow0$ in the
sense of tempered distribution. We can reach this goal by writing
down the formula of $\mathbf{B}^{-1}\left(\epsilon\right)$ explicitly
and examining the asymptotics of the Bessel functions near the origin
(detailed computations are given in the appendix). As a result, we
find that $\zeta=\left(1,1\right)^{\top}$ will serve the purpose,
in which case $\mu_{\epsilon}^{x}\rightarrow2\delta_{x}$ as $\epsilon\downarrow0$
for every $x\in\mathbb{R}^{4}$. From now on, we will assume $\mu_{\epsilon}^{x}$
is as in (\ref{eq:mu_epsilon def}) with $\zeta=\left(1,1\right)^{\top}$.
Since $\mathcal{I}\left(h{}_{\mu_{\epsilon}^{x}}\right)=\zeta^{\top}\mathbf{B}^{-1}\left(\epsilon\right)V_{\epsilon}^{x}$,
we can transfer the results in Lemma \ref{lem:on V_t(x)} to the Gaussian
family $\left\{ \mathcal{I}\left(h{}_{\mu_{\epsilon}^{x}}\right):x\in\mathbb{R}^{4},\epsilon>0\right\} $. 
\begin{thm}
\label{thm:mu_epsilon}Define the positive function $G:r\in\left(0,\infty\right)\mapsto G\left(r\right)\in\left(0,\infty\right)$
by 
\begin{eqnarray}
\begin{split}G\left(r\right) & \equiv\left(-\frac{1}{4\pi^{2}}\right)\frac{2I_{1}\left(r\right)K_{1}\left(r\right)+2I_{2}\left(r\right)K_{0}\left(r\right)-1}{I_{1}^{2}\left(r\right)-I_{0}\left(r\right)I_{2}\left(r\right)}.\end{split}
\label{eq: variance G}
\end{eqnarray}
Then, we have\\
(1), given $x\in\mathbb{R}^{4}$ and $\epsilon_{1}\geq\epsilon_{2}>0$\emph{,
\begin{equation}
\mathbb{E}^{\mathcal{W}}\left[\mathcal{I}\left(h_{\mu_{\epsilon_{1}}^{x}}\right)\mathcal{I}\left(h_{\mu_{\epsilon_{2}}^{x}}\right)\right]=\mathbb{E}^{\mathcal{W}}\left[\mathcal{I}^{2}\left(h_{\mu_{\epsilon_{1}}^{x}}\right)\right]=G\left(\epsilon_{1}\right).\label{eq:cov concentric}
\end{equation}
}In particular, $\left\{ \mathcal{I}\left(h_{\mu_{\epsilon}^{x}}\right):\epsilon>0\right\} $
is a Gaussian process with independent reversed increments in the
sense that $\mathcal{I}\left(h_{\mu_{\epsilon_{2}}^{x}}\right)-\mathcal{I}\left(h_{\mu_{\epsilon_{1}}^{x}}\right)$
is independent of $\sigma\left\{ \mathcal{I}\left(h_{\mu_{\eta}^{x}}\right):\eta\geq\epsilon_{1}\right\} $.\\
(2), given $x,y\in\mathbb{R}^{4}$, $x\neq y$, and $\epsilon_{1},\epsilon_{2}>0$
with $\epsilon_{1}>\left|x-y\right|+\epsilon_{2}$, \emph{
\begin{equation}
\begin{split}\mathbb{E}^{\mathcal{W}}\left[\mathcal{I}\left(h_{\mu_{\epsilon_{1}}^{x}}\right)\mathcal{I}\left(h_{\mu_{\epsilon_{2}}^{y}}\right)\right]=I_{0}\left(\left|x-y\right|\right)G\left(\epsilon_{1}\right)-\frac{1}{4\pi^{2}}\frac{I_{2}\left(\left|x-y\right|\right)}{I_{1}^{2}\left(\epsilon_{1}\right)-I_{0}\left(\epsilon_{1}\right)I_{2}\left(\epsilon_{1}\right)}.\end{split}
\label{eq:cov inclusion}
\end{equation}
}(3), given $x,y\in\mathbb{R}^{4}$, $x\neq y$, and $\epsilon_{1},\epsilon_{2}>0$
with $\left|x-y\right|>\epsilon_{1}+\epsilon_{2}$,\emph{
\begin{equation}
\mathbb{E}^{\mathcal{W}}\left[\mathcal{I}\left(h_{\mu_{\epsilon_{1}}^{x}}\right)\mathcal{I}\left(h_{\mu_{\epsilon_{2}}^{y}}\right)\right]=\frac{1}{2\pi^{2}}K_{0}\left(\left|x-y\right|\right).\label{eq:cov nonoverlap}
\end{equation}
}
\end{thm}
There is not much to be said about the proof since everything follows
from straightforward computations and the results in Lemma \ref{lem:on V_t(x)}.
However, we should point out the following facts. 

First, it is an easy matter to check that $G:\left(0,\infty\right)\rightarrow\left(0,\infty\right)$
is smooth and strictly decreasing with $\lim_{r\downarrow0}G\left(r\right)=+\infty$
and $\lim_{r\uparrow\infty}G\left(r\right)=0$. Therefore, $G^{-1}$
is defined and also strictly decreasing on $\left(0,\infty\right)$.
Fix a positive constant $R$, and for every $r\in(0,R]$, define 
\begin{equation}
0\leq t\equiv G\left(r\right)-G\left(R\right)\mbox{ and }X_{t}\equiv\mathcal{I}\left(h_{\mu_{G^{-1}\left(t+G\left(R\right)\right)}^{x}}\right)-\mathcal{I}\left(h_{\mu_{R}^{x}}\right).\label{eq:def of X_t}
\end{equation}
Then $\left\{ X_{t}:t\geq0\right\} $, as a Gaussian process on $\Theta$
under $\mathcal{W}$, has the same distribution as the standard Brownian
motion, which in particular is independent of the choice of $x$.
In other words, $\left\{ \mathcal{I}\left(h{}_{\mu_{\epsilon}^{x}}\right):0<\epsilon\leq R\right\} $
has the distribution of a Brownian motion up to a non-random time
change.

Second, the formulas (\ref{eq:cov inclusion}) and (\ref{eq:cov nonoverlap})
indicate that the covariance function does not depend on $\epsilon_{2}$
when $\overline{B_{\epsilon_{2}}\left(y\right)}\subseteq B_{\epsilon_{1}}\left(x\right)$
and does not even depend on $\epsilon_{1}$ or $\epsilon_{2}$ when
$\overline{B_{\epsilon_{2}}\left(y\right)}\cap\overline{B_{\epsilon_{1}}\left(x\right)}=\emptyset$.
If one does the calculation with circular averages of the two dimensional
GFF in each case analogously, one would see the same phenomenon. Namely,
the smaller radius does not appear in the covariance function if one
circle is entirely contained in the disk bounded by the other circle,
while neither of the radii matters if the two disks bounded by the
two circles respectively don't intersect. Such properties are consequences
of the mean value theorem applied to the Green's function of the Laplace
operator $\Delta$ in two dimensions, and these special properties
of harmonic functions are no longer available to us in four dimensions.
Nonetheless, we have seen from the above that by substituting$\left\{ \mathcal{I}\left(h{}_{\mu_{\epsilon}^{x}}\right):x\in\mathbb{R}^{4},\epsilon>0\right\} $
for the circular average, we will recover in the four-dimensional
setting properties similar to those in two dimensions.

Finally, by examining the asymptotics of the Bessel functions near
the origin, one finds that function $G$ as defined in (\ref{eq: variance G})
is asymptotic to $-\frac{1}{2\pi^{2}}\log r$ when $r$ is small.
Therefore, in both case (1) and case (2) from above, the covariance
function is asymptotic to $-\frac{1}{2\pi^{2}}\log\epsilon_{1}$ for
sufficiently small $\epsilon_{1}$, while in case (3), the right hand
side of (\ref{eq:cov nonoverlap}) is asymptotic to $-\frac{1}{2\pi^{2}}\log\left|x-y\right|$
for sufficiently small $\left|x-y\right|$. In this sense, the covariance
function of $\left\{ \mathcal{I}\left(h{}_{\mu_{\epsilon}^{x}}\right):x\in\mathbb{R}^{4},\epsilon>0\right\} $
does have logarithmic growth near diagonal as one would have expected. 

By now, one should be able to believe that $\left\{ \mathcal{I}\left(h{}_{\mu_{\epsilon}^{x}}\right):x\in\mathbb{R}^{4},\epsilon>0\right\} $
is a reasonable replacement for the circular average of the 2D GFF
in order to construct a random measure on $\mathbb{R}^{4}$. Indeed,
the construction based on this Gaussian family will be carried out
in the next section. We close this section with an important observation
about $\mathcal{I}\left(h_{\mu_{\epsilon}^{x}}\right)$ as a mapping
from the variable $x\in\mathbb{R}^{4}$ to a random variable on $\Theta$
under $\mathcal{W}$.

\begin{cor}
\label{cor:a.e. continuity in x}Given $\epsilon>0$, the mapping
\textup{$x\in\mathbb{R}^{4}\mapsto\mathcal{I}\left(h_{\mu_{\epsilon}^{x}}\right)\in L^{2}\left(\mathcal{W}\right)$}
is continuous. Moreover, if $\alpha\in\left(0,\frac{1}{2}\right)$,
then for almost every $\theta\in\Theta$ , \textup{$x\in\mathbb{R}^{4}\mapsto\mathcal{I}\left(h_{\mu_{\epsilon}^{x}}\right)\left(\theta\right)\in\mathbb{R}$}
is Hölder continuous with Hölder constant $\alpha$.\end{cor}
\begin{proof}
By Kolmogorov's continuity criterion (\cite{probability}, §4.3) applied
to Gaussian random variables, in order to prove both statements in
Corollary \ref{cor:a.e. continuity in x}, it would be sufficient
if  we can show that there exist constant $\beta>0$ and $0<C_{\beta,\epsilon}<\infty$
such that for every $x,y\in\mathbb{R}^{4}$, 
\[
\left\Vert \mathcal{I}\left(h_{\mu_{\epsilon}^{x}}\right)-\mathcal{I}\left(h_{\mu_{\epsilon}^{y}}\right)\right\Vert _{L^{2}\left(\mathcal{W}\right)}^{2}\leq C_{\beta,\epsilon}\left|x-y\right|^{\beta}.
\]
To simplify the notations, we write $\mu_{\epsilon}^{x}$ as $\mu_{\epsilon}^{x}\equiv f_{1}\left(\epsilon\right)\sigma_{\epsilon}^{x}+f_{2}\left(\epsilon\right)d\sigma_{\epsilon}^{x}$,
where both $f_{1}$ and $f_{2}$ are actually analytic functions in
$\epsilon\in[0,R]$ (the explicit formulas for $f_{1}$ and $f_{2}$
are given by (\ref{eq:f1 and f2}) in the appendix). Therefore, we
only need to show that both $\left\Vert \sigma_{\epsilon}^{x}-\sigma_{\epsilon}^{y}\right\Vert _{H^{-2}}^{2}$
and $\left\Vert d\sigma_{\epsilon}^{x}-d\sigma_{\epsilon}^{y}\right\Vert _{H^{-2}}^{2}$
are bounded by $C_{\beta,\epsilon}\left|x-y\right|^{\beta}$. Perhaps
the most straightforward way to see this is writing down the integral
expressions for $\left\Vert \sigma_{\epsilon}^{x}-\sigma_{\epsilon}^{y}\right\Vert _{H^{-2}}^{2}$
and $\left\Vert d\sigma_{\epsilon}^{x}-d\sigma_{\epsilon}^{y}\right\Vert _{H^{-2}}^{2}$.
Again, we refer to the formulas (\ref{eq:var sphe ave integral form})-(\ref{eq:cov sphe ave deriv integral form})
in the appendix. From there, together with the series expression for
the Bessel functions (\cite{BesselFunctions}, §3.1), it is an easy
matter to check that one can get the desired upper bound for $\left\Vert \sigma_{\epsilon}^{x}-\sigma_{\epsilon}^{y}\right\Vert _{H^{-2}}^{2}$
and $\left\Vert d\sigma_{\epsilon}^{x}-d\sigma_{\epsilon}^{y}\right\Vert _{H^{-2}}^{2}$
so long as $\beta\in\left(0,1\right)$. 
\end{proof}

\section{Construction of Random Measure}

In this section, we will use the Gaussian family $\left\{ \mathcal{I}\left(h_{\mu_{\epsilon}^{x}}\right):x\in\mathbb{R}^{4},\epsilon>0\right\} $
to construct a random measure on $\mathbb{R}^{4}$ which ``formally''
can be written as ``$m\left(dx\right)=e^{2\gamma\theta\left(x\right)}dx$''
where $\theta\in\Theta$ is chosen under the distribution of $\mathcal{W}$,
$\gamma\geq0$ is a constant and $dx$ is the Lebesgue measure on
$\mathbb{R}^{4}$. Recall that at every $x\in\mathbb{R}^{4}$, $\mu_{\epsilon}^{x}$
tends to $2\delta_{x}$ as $\epsilon\downarrow0$ in the sense of
tempered distribution, so we can take the value of the random variable
$\theta\mapsto\frac{1}{2}\mathcal{I}\left(h_{\mu_{\epsilon}^{x}}\right)\left(\theta\right)$
as the approximation for ``$\theta\left(x\right)$'' as $\epsilon\downarrow0$.
Furthermore, Corollary \ref{cor:a.e. continuity in x} certainly guarantees
that with $\epsilon$ fixed, one can always make the mapping 
\[
\left(x,\theta\right)\in\mathbb{R}^{4}\times\Theta\mapsto\mathcal{I}\left(h_{\mu_{\epsilon}^{x}}\right)\left(\theta\right)\in\mathbb{R}
\]
measurable with respect to $\mathfrak{B}_{\mathbb{R}^{4}}\times\mathfrak{B}_{\Theta}$.
In addition, we can assume for every $\theta\in\Theta$, the function
\[
x\in\mathbb{R}^{4}\mapsto E_{\epsilon}^{\theta}\left(x\right)\equiv\exp\left(\gamma\mathcal{I}\left(h_{\mu_{\epsilon}^{x}}\right)\left(\theta\right)-\frac{\gamma^{2}}{2}G\left(\epsilon\right)\right)
\]
is positive and continuous, and hence if $m_{\epsilon}^{\theta}\left(dx\right)\equiv E_{\epsilon}^{\theta}\left(x\right)dx$,
then $m_{\epsilon}^{\theta}\left(dx\right)$ is a positive regular
and $\sigma-$finite Borel measure on $\mathbb{R}^{4}$. Moreover,
given any Borel $B\subseteq\mathbb{R}^{4}$, the mapping 
\[
\theta\in\Theta\mapsto m_{\epsilon}^{\theta}\left(B\right)=\int_{B}E_{\epsilon}^{\theta}\left(x\right)dx\in\mathbb{R}
\]
is also non-negative and measurable. Hence by Tonelli's Theorem and
the fact that $\mathbb{E}^{\mathcal{W}}\left[E_{\epsilon}^{\theta}\left(x\right)\right]=1$
for every $x\in\mathbb{R}^{4}$ and $\epsilon>0$,
\[
\mathbb{E}^{\mathcal{W}}\left[m_{\epsilon}^{\theta}\left(B\right)\right]=\int_{B}\mathbb{E}^{\mathcal{W}}\left[E_{\epsilon}^{\theta}\left(x\right)\right]dx=\mbox{vol}\left(B\right).
\]

Since $m_{\epsilon}^{\theta}\left(dx\right)$ is simply the Lebesgue
measure on $\mathbb{R}^{4}$ when $\gamma=0$, from now on we will
only consider the case when $\gamma>0$. We want to study the convergence
of $m_{\epsilon}^{\theta}\left(dx\right)$ as $\epsilon\downarrow0$.
In order to have the desired convergence, we only consider $\epsilon$
taking values along a sequence $\left\{ \epsilon_{n}\equiv\epsilon_{0}^{n}:\, n\geq1\right\} $
for some fixed $\epsilon_{0}\in\left(0,1\right)$. Without loss of
generality, we will assume $m_{\epsilon_{n}}^{\theta}\left(dx\right)$
is well defined as above for all $n\geq1$ and every $\theta\in\Theta$.
For the sake of convenience in later discussions, we will abuse the
notation by identifying ``$m_{\epsilon_{0}}^{\theta}\left(dx\right)$''
with 0. We want to show that as $n\rightarrow\infty$, almost surely
the sequence $\left\{ m_{\epsilon_{n}}^{\theta}\left(dx\right)\right\} $
converges weakly to a limit measure $m^{\theta}\left(dx\right)$,
written as $m_{\epsilon_{n}}^{\theta}\left(dx\right)\rightharpoonup m^{\theta}\left(dx\right)$,
in the sense that $\int_{\mathbb{R}^{4}}f\left(x\right)m_{\epsilon_{n}}^{\theta}\left(dx\right)$
converges to $\int_{\mathbb{R}^{4}}f\left(x\right)m^{\theta}\left(dx\right)$
for every $f\in C_{c}\left(\mathbb{R}^{4}\right)$ which is the space
of continuous and compactly supported functions on $\mathbb{R}^{4}$.
To reach this goal, it suffices to show the convergence of $\int_{\mathbb{R}^{4}}f\left(x\right)m_{\epsilon_{n}}^{\theta}\left(dx\right)$
when $f$ is any continuous function on $\Gamma$ for any given compact
set $\Gamma\subseteq\mathbb{R}^{4}$. In fact, we have the following
result that holds for more general $f$ so long as $f$ is bounded
and measurable on $\Gamma$. 
\begin{lem}
\label{lem:conv of total mass}Assume $0<\gamma^{2}<2\pi^{2}$ and
$\Gamma\subseteq\mathbb{R}^{4}$ is compact. There exists a square
integrable random variable $\theta\in\Theta\mapsto\overline{m^{\theta}}\left(\Gamma\right)\in\mathbb{R}^{+}$
such that 
\[
\sum_{n=0}^{N-1}\left|m_{\epsilon_{n+1}}^{\theta}\left(\Gamma\right)-m_{\epsilon_{n}}^{\theta}\left(\Gamma\right)\right|\mbox{ converges to }\overline{m^{\theta}}\left(\Gamma\right)\mbox{ as }N\rightarrow\infty
\]
almost surely as well as in $L^{2}\left(\mathcal{W}\right)$. 

Similarly, for every bounded measurable function $f$ with $\mbox{supp}\left(f\right)\subseteq\Gamma$,
there exists $M^{\theta}\left(f\right)\in L^{2}\left(\mathcal{W}\right)$
such that 
\[
M_{\epsilon_{n}}^{\theta}\left(f\right)\equiv\int_{\mathbb{R}^{4}}f\left(x\right)m_{\epsilon_{n}}^{\theta}\left(dx\right)\mbox{ converges to }M^{\theta}\left(f\right)\mbox{ as }n\rightarrow\infty
\]
almost surely as well as in $L^{2}\left(\mathcal{W}\right)$, and
$\left|M^{\theta}\left(f\right)\right|\leq\overline{m^{\theta}}\left(\Gamma\right)\left\Vert f\right\Vert _{u}$
almost surely.\end{lem}
\begin{proof}
To prove the first statement, we rewrite $\left(m_{\epsilon_{n+1}}^{\theta}\left(\Gamma\right)-m_{\epsilon_{n}}^{\theta}\left(\Gamma\right)\right)^{2}$,
$n\geq1$, as the following double integral: 
\[
\iint_{\Gamma^{2}}\left[E_{\epsilon_{n+1}}^{\theta}\left(x\right)E_{\epsilon_{n+1}}^{\theta}\left(y\right)+E_{\epsilon_{n}}^{\theta}\left(x\right)E_{\epsilon_{n}}^{\theta}\left(y\right)-2E_{\epsilon_{n+1}}^{\theta}\left(x\right)E_{\epsilon_{n}}^{\theta}\left(y\right)\right]dxdy.
\]
By Tonelli's Theorem, the $\mathcal{W}$-expectation of above equals
\[
\iint_{\Gamma^{2}}\left\{ e^{\gamma^{2}\mathbb{E}\left[\mathcal{I}\left(h_{\mu_{\epsilon_{n+1}}^{x}}\right)\mathcal{I}\left(h_{\mu_{\epsilon_{n+1}}^{y}}\right)\right]}+e^{\gamma^{2}\mathbb{E}\left[\mathcal{I}\left(h_{\mu_{\epsilon_{n}}^{x}}\right)\mathcal{I}\left(h_{\mu_{\epsilon_{n}}^{y}}\right)\right]}-2e^{\gamma^{2}\mathbb{E}\left[\mathcal{I}\left(h_{\mu_{\epsilon_{n+1}}^{x}}\right)\mathcal{I}\left(h_{\mu_{\epsilon_{n}}^{y}}\right)\right]}\right\} dxdy.
\]
We split this integral by dividing the domain into two parts: 
\[
\iint_{\left|x-y\right|>2\epsilon_{n}}\mbox{ and \ensuremath{\iint_{0\leq\left|x-y\right|\leq2\epsilon_{n}}} .}
\]
The formula (\ref{eq:cov nonoverlap}) implies the integrand is always
zero in the designated domain of the first part. As for the second
part, the integrand is always bounded by $2e^{\gamma^{2}G\left(\epsilon_{n+1}\right)}$
while the volume of the integral domain is bounded by $C\epsilon_{n}^{4}$
for some constant%
\footnote{Throughout this section, ``$C$'' denotes a positive finite constant
that is universal in $\epsilon_{n}$. The value of $C$ may change
from line to line.%
} $C$. Together with the observations made in Section 2 about the
asymptotics of $G$, one finds that 
\begin{equation}
\begin{split}\mathbb{E}^{\mathcal{W}}\left[\left|m_{\epsilon_{n+1}}^{\theta}\left(\Gamma\right)-m_{\epsilon_{n}}^{\theta}\left(\Gamma\right)\right|^{2}\right] & \leq Ce^{-\left(8\pi^{2}-\gamma^{2}\right)G\left(\epsilon_{n}\right)}.\end{split}
\label{eq:total mass estimate of 2nd moment of difference}
\end{equation}
The square root of the right hand side of (\ref{eq:total mass estimate of 2nd moment of difference})
is summable in $n\geq1$ and meanwhile $m_{\epsilon_{1}}^{\theta}\left(\Gamma\right)$
is clearly square integrable, so 
\begin{equation}
\overline{m^{\theta}}\left(\Gamma\right)\equiv\sum_{n=0}^{\infty}\left|m_{\epsilon_{n+1}}^{\theta}\left(\Gamma\right)-m_{\epsilon_{n}}^{\theta}\left(\Gamma\right)\right|\label{eq:def of total mass}
\end{equation}
is square integrable and the convergence takes place in $L^{2}\left(\mathcal{W}\right)$.
Furthermore, the series $\sum_{n=0}^{N}\left|m_{\epsilon_{n+1}}^{\theta}\left(\Gamma\right)-m_{\epsilon_{n}}^{\theta}\left(\Gamma\right)\right|$
also converges to $\overline{m^{\theta}}\left(\Gamma\right)$ almost
surely along a subsequence, but since the series is monotonic in $N$,
it must converge to $\overline{m^{\theta}}\left(\Gamma\right)$ almost
surely along the full sequence.

The second statement of the lemma follows from the same arguments.
In fact, given a bounded measurable function $f$ with $\mbox{supp}\left(f\right)\subseteq\Gamma$,
if one replaces $m_{\epsilon_{n}}^{\theta}\left(\Gamma\right)$ by
$M_{\epsilon_{n}}^{\theta}\left(f\right)$ in every step of the proof
above, one can see that $\sum_{n=0}^{\infty}\left|M_{\epsilon_{n+1}}^{\theta}\left(f\right)-M_{\epsilon_{n}}^{\theta}\left(f\right)\right|$
is also square integrable. The rest of the proof is straightforward.\end{proof}
\begin{thm}
\label{thm: construction of the measure}Assume $0<\gamma^{2}<2\pi^{2}$.
For almost every $\theta\in\Theta$, there exists a non-negative regular
$\sigma-$finite Borel measure $m^{\theta}\left(dx\right)$ on $\mathbb{R}^{4}$
such that 
\[
m_{\epsilon_{n}}^{\theta}\left(dx\right)\rightharpoonup m^{\theta}\left(dx\right)\mbox{ as }n\rightarrow\infty,
\]
and for every compact set $\Gamma\subseteq\mathbb{R}^{4}$, $\left\Vert m^{\theta}\right\Vert _{\mbox{var},\Gamma}\leq\overline{m^{\theta}}\left(\Gamma\right)$
where $\left\Vert m^{\theta}\right\Vert _{\mbox{var},\Gamma}$ is
the total variation of $m^{\theta}\left(dx\right)$ over $\Gamma$
and $\overline{m^{\theta}}\left(\Gamma\right)$ is as defined in (\ref{eq:def of total mass}).

In particular, for every $f\in C_{c}\left(\mathbb{R}^{4}\right)$,
\[
\int_{\mathbb{R}^{4}}f\left(x\right)m_{\epsilon_{n}}^{\theta}\left(dx\right)\mbox{ converges to }\int_{\mathbb{R}^{4}}f\left(x\right)m^{\theta}\left(dx\right)\mbox{ as }n\rightarrow\infty
\]
almost surely as well as in $L^{2}\left(\mathcal{W}\right)$.\end{thm}
\begin{proof}
Clearly we only need to prove the first statement of the theorem,
because assuming $m_{\epsilon_{n}}^{\theta}\left(dx\right)\rightharpoonup m^{\theta}\left(dx\right)$
almost surely, the second statement is simply repeating the second
result in Lemma \ref{lem:conv of total mass} with $M^{\theta}\left(f\right)=\int_{\mathbb{R}^{4}}f\left(x\right)m^{\theta}\left(dx\right)$
for $f\in C_{c}\left(\mathbb{R}^{4}\right)$. As we mentioned earlier,
to obtain the limit measure $m^{\theta}\left(dx\right)$, it suffices
to show the convergence of $m_{\epsilon_{n}}^{\theta}\left(dx\right)$
on any compact set $\Gamma\subseteq\mathbb{R}^{4}$. We will achieve
this goal via the Riesz representation theorem. We have already seen
from the second part of Lemma \ref{lem:conv of total mass} that,
if denote $M_{\epsilon_{n}}^{\theta}\left(f\right)\equiv\int_{\mathbb{R}^{4}}f\left(x\right)m_{\epsilon_{n}}^{\theta}\left(dx\right)$
for every $n\geq1$ and every bounded measurable function $f$ supported
on $\Gamma$, then 
\begin{equation}
M^{\theta}\left(f\right)\equiv\lim_{n\rightarrow\infty}M_{\epsilon_{n}}^{\theta}\left(f\right)\mbox{ exists and }\left|M^{\theta}\left(f\right)\right|\leq\left\Vert f\right\Vert _{u}\overline{m^{\theta}}\left(\Gamma\right)<\infty\label{eq: Riesz criterion}
\end{equation}
with probability one. However, to get the almost sure existence of
$m^{\theta}\left(dx\right)$, we need to argue that with probability
one, the statement above holds simultaneously for all functions $f$
from a ``large enough'' class. To this end, we make use of the separability
of the Banach space $C\left(\Gamma\right)$ (when equipped with the
uniform norm $\left\Vert \cdot\right\Vert _{u}$). Namely, we can
choose a countable sequence $\left\{ f_{k}:k\geq1\right\} $ which
is a dense subset of $C\left(\Gamma\right)$, so almost surely the
statement (\ref{eq: Riesz criterion}) holds%
\footnote{In this discussion, we will simply assume $f\equiv0$ outside $\Gamma$
for every $f\in C\left(\Gamma\right)$, so $M^{\theta}\left(f\right)$
and $M_{\epsilon_{n}}^{\theta}\left(f\right)$ are still well defined.%
} simultaneously for all $f_{k}$, $k\geq1$. 

Now let's focus on $\theta\in\Theta$ such that (\ref{eq: Riesz criterion})
holds for all $f_{k}$, $k\geq1$. Given a general $f\in C\left(\Gamma\right)$,
let $\left\{ f_{k_{j}}:j\geq1\right\} $ be a subsequence such that
$f_{k_{j}}\rightarrow f$ in $\left\Vert \cdot\right\Vert _{u}$ as
$j\rightarrow\infty$, then for every $l,n\geq1$, 
\begin{eqnarray*}
\left|M_{\epsilon_{l}}^{\theta}\left(f\right)-M_{\epsilon_{n}}^{\theta}\left(f\right)\right| & \leq & \left|M_{\epsilon_{l}}^{\theta}\left(f\right)-M_{\epsilon_{l}}^{\theta}\left(f_{k_{j}}\right)\right|+\left|M_{\epsilon_{l}}^{\theta}\left(f_{k_{j}}\right)-M_{\epsilon_{n}}^{\theta}\left(f_{k_{j}}\right)\right|\\
 &  & \qquad\qquad\qquad\qquad\qquad\qquad\qquad+\left|M_{\epsilon_{n}}^{\theta}\left(f_{k_{j}}\right)-M_{\epsilon_{n}}^{\theta}\left(f\right)\right|\\
 & \leq & 2\overline{m^{\theta}}\left(\Gamma\right)\left\Vert f-f_{k_{j}}\right\Vert _{u}+\left|M_{\epsilon_{l}}^{\theta}\left(f_{k_{j}}\right)-M_{\epsilon_{n}}^{\theta}\left(f_{k_{j}}\right)\right|.
\end{eqnarray*}
Obviously $\left\{ M_{\epsilon_{n}}^{\theta}\left(f\right):n\geq1\right\} $
forms a Cauchy sequence which immediately implies that (\ref{eq: Riesz criterion})
also holds for $f$ and in addition $M^{\theta}\left(f\right)=\lim_{j\rightarrow\infty}M^{\theta}\left(f_{k_{j}}\right)$.
Therefore, for almost every $\theta\in\Theta$, $f\mapsto M^{\theta}\left(f\right)$
is a linear and bounded functional on $C$$\left(\Gamma\right)$,
which, by the Riesz representation theorem, gives rise to a unique
regular Borel measure $m^{\theta}\left(dx\right)$ on $\Gamma$ such
that $M^{\theta}\left(f\right)=\int_{\Gamma}f\left(x\right)m^{\theta}\left(dx\right)$
for all $f\in C\left(\Gamma\right)$ and the total variation of $m^{\theta}\left(dx\right)$
is equal to the operator norm of $M^{\theta}$ which is bounded by
$\overline{m^{\theta}}\left(\Gamma\right)$. It's also clear that
$m^{\theta}\left(dx\right)$ is non-negative and $m_{\epsilon_{n}}^{\theta}\left(dx\right)\rightharpoonup m^{\theta}\left(dx\right)$. 
\end{proof}
Compared with the second statement in Lemma \ref{lem:conv of total mass},
the second statement in Theorem \ref{thm: construction of the measure}
seems ``weaker'' since we have restricted ourselves to $f\in C_{c}\left(\mathbb{R}^{4}\right)$.
However, we point out that the same statement, i.e., $M_{\epsilon_{n}}^{\theta}\left(f\right)$
converges to $\int_{\mathbb{R}^{4}}f\left(x\right)m^{\theta}\left(dx\right)$
both almost surely and in $L^{2}\left(\mathcal{W}\right)$, is no
longer true if $f$ is only assumed to be a bounded measurable function
with compact support. The reason is the following: for such a function
$f$, although the existence of $M^{\theta}\left(f\right)=\lim_{n\rightarrow\infty}M_{\epsilon_{n}}^{\theta}\left(f\right)$
is guaranteed by Lemma \ref{lem:conv of total mass}, in general one
cannot draw any conclusion on the relation between $M^{\theta}\left(f\right)$
and $\int_{\mathbb{R}^{4}}f\left(x\right)m^{\theta}\left(dx\right)$,
because $m_{\epsilon_{n}}^{\theta}\left(dx\right)$ only converges
to $m^{\theta}\left(dx\right)$ weakly and one does not have control
over $\left\Vert m_{\epsilon_{n}}^{\theta}-m^{\theta}\right\Vert _{\mbox{var},\mbox{supp}\left(f\right)}$.
However, under some circumstances, we can derive a relation between
the two random variables $M^{\theta}\left(f\right)$ and $\int_{\mathbb{R}^{4}}f\left(x\right)m^{\theta}\left(dx\right)$.
For example, if $f=\chi_{A}$ is the indicator function of a bounded
open set $A\subseteq\mathbb{R}^{4}$, then the weak convergence result
implies $m^{\theta}\left(A\right)\leq\lim_{n\rightarrow\infty}m_{\epsilon_{n}}^{\theta}\left(A\right)$
almost surely. Meanwhile, the $L^{2}\left(\mathcal{W}\right)$ convergence
of $m_{\epsilon_{n}}^{\theta}\left(A\right)$ certainly leads to 
\[
\mathbb{E}^{\mathcal{W}}\left[\lim_{n\rightarrow\infty}m_{\epsilon_{n}}^{\theta}\left(A\right)\right]=\lim_{n\rightarrow\infty}\mathbb{E}^{\mathcal{W}}\left[m_{\epsilon_{n}}^{\theta}\left(A\right)\right]=\mbox{vol}\left(A\right);
\]
on the other hand, let $\left\{ f_{l}:l\geq1\right\} \subseteq C_{c}\left(\mathbb{R}^{4}\right)$
be a sequence such that $0\leq f_{l}\nearrow\chi_{A}$ as $l\rightarrow\infty$,
then by the monotone convergence theorem and the second statement
in Theorem \ref{thm: construction of the measure}, 
\[
\mathbb{E}^{\mathcal{W}}\left[m^{\theta}\left(A\right)\right]=\lim_{l\rightarrow\infty}\mathbb{E}^{\mathcal{W}}\left[\int_{\mathbb{R}^{4}}f_{l}\left(x\right)m^{\theta}\left(dx\right)\right]=\lim_{l\rightarrow\infty}\int_{\mathbb{R}^{4}}f_{l}\left(x\right)dx=\mbox{vol}\left(A\right).
\]
This can only be possible if $m^{\theta}\left(A\right)=\lim_{n\rightarrow\infty}m_{\epsilon_{n}}^{\theta}\left(A\right)$
almost surely. More generally (\cite{probability}, §9.1), if $f$
is bounded and upper semicontinuous (or lower semicontinuous or $m^{\theta}-$almost
surely continuous), then it follows from a similar argument that $M^{\theta}\left(f\right)=\lim_{n\rightarrow\infty}M_{\epsilon_{n}}^{\theta}\left(f\right)$
almost surely, so the second statement in Theorem \ref{thm: construction of the measure}
also holds for $f$.

The fact as stated above that $\mathbb{E}^{\mathcal{W}}\left[m^{\theta}\left(A\right)\right]=\mbox{vol}\left(A\right)$
for every bounded open set guarantees that the limit measure $m^{\theta}\left(dx\right)$
cannot be almost everywhere trivial, i.e., $\mathcal{W}\left(m^{\theta}\left(dx\right)=0\right)<1$.
In fact, we will prove later (in Lemma \ref{lem: tail estimate})
that $\mathcal{W}\left(m^{\theta}\left(dx\right)=0\right)=0$, so
$m^{\theta}\left(dx\right)$ is almost surely a positive measure.
On the other hand, the following simple observation shows that $m^{\theta}\left(dx\right)$
almost surely does not assign positive mass to any given point. To
see this, recall the assumption $0<\gamma^{2}<2\pi^{2}$ and the fact
that $G\left(r\right)$ is asymptotic to $-\frac{1}{2\pi^{2}}\log r$
when $r$ is small. Then it is an easy matter to check that for any
fixed $x\in\mathbb{R}^{4}$, 
\[
\mathbb{E}^{\mathcal{W}}\left[\limsup_{n\rightarrow\infty}e^{4\gamma^{2}G\left(\epsilon_{n}\right)}m^{\theta}\left(\overline{B_{\epsilon_{n}}\left(x\right)}\right)\right]=0.
\]
Therefore, if denote 
\begin{equation}
\Theta_{x}\equiv\left\{ \theta\in\Theta:\lim_{n\rightarrow\infty}e^{4\gamma^{2}G\left(\epsilon_{n}\right)}m^{\theta}\left(\overline{B_{\epsilon_{n}}\left(x\right)}\right)=0\right\} ,\label{eq:def of Theta_x}
\end{equation}
then $\Theta_{x}$ is clearly a measurable subset of $\Theta$ and
$\mathcal{W}\left(\Theta_{x}\right)=1$. 

We will close this section by a remark about the condition of the
constant%
\footnote{In this article, we don't have particular emphasis on the potential
physics meaning of $\gamma$.%
} $\gamma$. Readers may have noticed that the constraint $0<\gamma^{2}<2\pi^{2}$
in Lemma \ref{lem:conv of total mass} and Theorem \ref{thm: construction of the measure}
is more than what the proofs require. However, one reason of having
this condition on $\gamma$ is that it guarantees the proof of Lemma
\ref{lem:conv of total mass} being correct even if one replaces $\gamma$
by $2\gamma$. In other words, if we denote 
\[
m_{\epsilon_{n}}^{\theta,2\gamma}\left(dx\right)\equiv e^{2\gamma\mathcal{I}\left(h_{\mu_{\epsilon_{n}}^{x}}\right)\left(\theta\right)-2\gamma^{2}G\left(\epsilon_{n}\right)}dx
\]
and define $\overline{m^{\theta,2\gamma}}\left(\Gamma\right)$ similarly
using $m_{\epsilon_{n}}^{\theta,2\gamma}\left(\Gamma\right)$ for
any compact set $\Gamma\subseteq\mathbb{R}^{4}$, then $\overline{m^{\theta,2\gamma}}\left(\Gamma\right)$
is also square integrable and in particular $\overline{m^{\theta,2\gamma}}\left(\Gamma\right)$
is almost surely finite. Some proofs in Section 5 make use of this
consideration and hence the condition $0<\gamma^{2}<2\pi^{2}$ becomes
necessary there. We will remind readers when it comes to those situations.

\section{KPZ Relation}

Throughout later discussions, we will always assume $0<\gamma^{2}<2\pi^{2}$
and for every $\theta\in\Theta$, $m^{\theta}\left(dx\right)$ is
a non-negative regular and $\sigma-$finite Borel measure on $\mathbb{R}^{4}$
and $m_{\epsilon_{n}}^{\theta}\left(dx\right)\rightharpoonup m^{\theta}\left(dx\right)$
(otherwise one simply assigns $m_{\epsilon_{n}}^{\theta}\left(dx\right)=m^{\theta}\left(dx\right)=dx$
for all $n\geq1$ on a measurable null set of $\Theta$). In this
section, we would like to establish a KPZ relation between the Euclidean
scaling exponent of a bounded (fractal) Borel set on $\mathbb{R}^{4}$
and its counterpart under the random measure $m^{\theta}\left(dx\right)$.
We first recall from \cite{DS1} some definitions. Given a bounded
Borel $D\subseteq\mathbb{R}^{4}$, the constant $\kappa\in\left[0,1\right]$
is called the Euclidean \emph{scaling exponent} of $D$ if 
\[
\lim_{\lambda\downarrow0}\frac{\log\mbox{vol}\left(D_{\lambda}\right)}{\log\lambda^{4}}=\kappa,
\]
where $\lambda>0$ and $D_{\lambda}\equiv\cup_{x\in D}B_{\lambda}\left(x\right)$
is the canonical $\lambda-$neighborhood of $D$. In the random measure
setting, for every $\Lambda>0$, we first consider the mapping from
$\mathbb{R}^{4}\times\Theta$ to $\left[0,\infty\right]$:
\begin{equation}
\left(x,\theta\right)\mapsto r_{\Lambda}\left(x,\theta\right)\equiv\begin{cases}
\sup\left\{ r>0:m^{\theta}\left(B_{r}\left(x\right)\right)\leq\Lambda\right\}  & \mbox{if }\theta\in\Theta_{x},\\
0 & \mbox{otherwise,}
\end{cases}\label{eq:def of r_Delta under M}
\end{equation}
where $\Theta_{x}$ is as determined in (\ref{eq:def of Theta_x}),
then define the\emph{ isothermal $\Lambda-$neighborhood }of $D$
by 
\begin{equation}
D^{\Lambda,\theta}\equiv\left\{ x\in\mathbb{R}^{4}:\begin{split} & \mbox{either }r_{\Lambda}\left(x,\theta\right)>0\mbox{ and dist}\left(x,D\right)<r_{\Lambda}\left(x,\theta\right) & \mbox{ }\\
 & \mbox{or }r_{\Lambda}\left(x,\theta\right)=0\mbox{ and }x\in D
\end{split}
\right\} .\label{eq:vol-delta nbhd of D}
\end{equation}
Intuitively, $B_{r_{\Lambda}\left(x,\theta\right)}\left(x\right)$
is the ``largest'' ball (in radius) centered at $x$ with volume
$\Lambda$ under the random measure $m^{\theta}\left(dx\right)$,
and $D^{\Lambda,\theta}$ is the ``neighborhood'' obtained by covering
$D$ with all such balls with equal volume. If there exists constant
$K\in\left[0,1\right]$ such that 
\begin{equation}
\lim_{\Lambda\downarrow0}\frac{\log\mathbb{E}^{\mathcal{W}}\left[m^{\theta}\left(D^{\Lambda,\theta}\right)\right]}{\log\Lambda}=K,\label{eq:def of K}
\end{equation}
then $K$ is called the \emph{quantum scaling exponent} of $D$. $K$
can be viewed as the ``expected'' scaling exponent and as such the
counterpart of $\kappa$ under the random measure $m^{\theta}\left(dx\right)$.
In two dimensions, $\kappa$ and $K$ satisfy the so-called KPZ relation
which is a quadratic relation. Our goal in this section is to extend
such a relation to the four dimensional setting.

However, we haven't yet justified the ``meaning'' of the notations
in (\ref{eq:def of K}). First we have to check that $D^{\Lambda,\theta}$
is a Borel set in $\mathbb{R}^{4}$ for every $\theta\in\Theta$,
which requires us to verify the measurability of $\left(x,\theta\right)\mapsto r_{\Lambda}\left(x,\theta\right)$
with respect to $\mathfrak{B}_{\mathbb{R}^{4}}\times\mathfrak{B}_{\Theta}$.
To this end, we observe that $\left(x,\theta\right)\mapsto m^{\theta}\left(B_{r}\left(x\right)\right)$
is measurable for every $r>0$ because there certainly exists continuous
mapping $x\in\mathbb{R}^{4}\mapsto\rho_{l}^{x}\in C_{c}\left(\mathbb{R}^{4}\right)$
for every $l\geq1$ with $0\leq\rho_{l}^{x}\nearrow\chi_{B_{r}\left(x\right)}$
and hence $M^{\theta}\left(\rho_{l}^{x}\right)\nearrow m^{\theta}\left(B_{r}\left(x\right)\right)$
as $l\rightarrow\infty$ for every $\left(x,\theta\right)$. From
this we conclude that $\left\{ \left(x,\theta\right):\theta\in\Theta_{x}\right\} $
is a Borel set in $\mathbb{R}^{4}\times\Theta$, and then the measurability
of $r_{\Lambda}\left(x,\theta\right)$ follows simply by identifying
the set $\left\{ \left(x,\theta\right):0<r_{\Lambda}\left(x,\theta\right)\leq a\right\} $
with $\left\{ \left(x,\theta\right):\theta\in\Theta_{x},m^{\theta}\left(B_{a}\left(x\right)\right)\geq\Lambda\right\} $
for every $a>0$. Therefore, for every $\theta$, $D^{\Lambda,\theta}$
is a Borel set in $\mathbb{R}^{4}$ since it is the $\theta-$section
of the following Borel set in $\mathbb{R}^{4}\times\Theta$:
\[
\left(D\times\Theta\right)\cup\left\{ \left(x,\theta\right):r_{\Lambda}\left(x,\theta\right)>0,\mbox{ dist}\left(x,D\right)<r_{\Lambda}\left(x,\theta\right)\right\} .
\]
Next, we need to show that $\theta\mapsto m^{\theta}\left(D^{\Lambda,\theta}\right)$
is measurable with respect to $\mathfrak{B}_{\Theta}$. More generally,
if $C\in\mathfrak{B}_{\mathbb{R}^{4}}\times\mathfrak{B}_{\Theta}$
and $C^{\theta}\equiv\left\{ x\in\mathbb{R}^{4}:\left(x,\theta\right)\in C\right\} $
is the $\theta-$section of $C$, we claim that $\theta\mapsto m^{\theta}\left(C^{\theta}\right)$
is measurable with respect to $\mathfrak{B}_{\Theta}$. To see this,
denote $C_{N}^{\theta}\equiv C^{\theta}\cap B_{N}\left(0\right)$
for every $N\geq1$ and choose a sequence of mollifiers $\left\{ \varphi_{k}\in\left[0,1\right]:k\geq1\right\} \subseteq C_{c}\left(\mathbb{R}^{4}\right)$
such that $g_{k}^{\theta}\equiv\varphi_{k}\star\chi_{C_{N}^{\theta}}$
converges to $\chi_{C_{N}^{\theta}}$ under $\left\Vert \cdot\right\Vert _{u}$
as $k\rightarrow\infty$. Moreover, for every $k\geq1$, since $C\in\mathfrak{B}_{\mathbb{R}^{4}}\times\mathfrak{B}_{\Theta}$,
\[
\left(x,\theta\right)\mapsto g_{k}^{\theta}\left(x\right)=\int_{\left(y,\theta\right)\in C,\left|y\right|<N}\varphi_{k}\left(x-y\right)dy
\]
is also measurable with respect to $\mathfrak{B}_{\mathbb{R}^{4}}\times\mathfrak{B}_{\Theta}$.
Thus, because $g_{k}^{\theta}\in C_{c}\left(\mathbb{R}^{4}\right)$
and $\left\Vert g_{k}^{\theta}\right\Vert _{u}\leq1$ for every $k\geq1$
and $m_{\epsilon_{n}}^{\theta}\left(dx\right)\rightharpoonup m^{\theta}\left(dx\right)$,
\[
m^{\theta}\left(C^{\theta}\right)=\lim_{N\rightarrow\infty}\lim_{k\rightarrow\infty}\lim_{n\rightarrow\infty}\int_{\mathbb{R}^{4}}g_{k}^{\theta}\left(x\right)E_{\epsilon_{n}}^{\theta}\left(x\right)dx,
\]
which is a measurable function in $\theta$.

At this point, one can already tell that it is convenient to consider
the spatial variable $x$ and the GFF $\theta$ at the same time.
In particular, if $\mathcal{M}\left(dxd\theta\right)\equiv m^{\theta}\left(dx\right)\mathcal{W}\left(d\theta\right)$,
then it's clear from the preceding that $\mathcal{M}\left(dxd\theta\right)$
is a non-negative $\sigma-$finite Borel measure on the product space
$\mathbb{R}^{4}\times\Theta$. To connect $\kappa$ with $K$, it
is natural to investigate the distribution of $r_{\Lambda}\left(x,\theta\right)$
under the $\mathcal{M}\left(dxd\theta\right)$. To get started, we
will first look at the distribution of $m^{\theta}\left(B_{r}\left(x\right)\right)$
under $\mathcal{M}\left(dxd\theta\right)$ for any given $r>0$. The
following lemma says that we can change our ``perspective'' by first
choosing a ``base'' point $x\in\mathbb{R}^{4}$ and then examine
the distribution of $m^{\theta}\left(B_{r}\left(x\right)\right)$
at $x$. This is realized by a procedure of exchanging the order of
integration. The proof of Lemma \ref{lem:area dist under W*N} is
given in Section 5. 
\begin{lem}
\label{lem:area dist under W*N} For every $x\in\mathbb{R}^{4}$,
let $\Theta_{x}$ be the measurable subset of $\Theta$ as determined
in (\ref{eq:def of Theta_x}), and define 
\begin{equation}
\theta\mapsto\hat{m}^{\theta,x}\left(dy\right)\equiv\begin{cases}
\exp\left(\frac{\gamma^{2}}{2\pi^{2}}K_{0}\left(\left|x-y\right|\right)\right)m^{\theta}\left(dy\right) & \mbox{if }\theta\in\Theta_{x},\\
m^{\theta}\left(dy\right) & \mbox{otherwise.}
\end{cases}\label{eq:area dist under W*N}
\end{equation}
Then almost surely $\hat{m}^{\theta,x}\left(dy\right)$ is also a
non-negative regular and $\sigma-$finite Borel measure on $\mathbb{R}^{4}$.

Moreover, for every $r>0$, compact set $\Gamma\subseteq\mathbb{R}^{4}$
and $F\in C_{0}\left(\mathbb{R}^{4}\times[0,\infty)\right)$,
\begin{equation}
\begin{split} & \int_{\Theta}\int_{\Gamma}F\left(x,m^{\theta}\left(B_{r}\left(x\right)\right)\right)m^{\theta}\left(dx\right)\mathcal{W}\left(d\theta\right)=\int_{\Gamma}\int_{\Theta}F\left(x,\hat{m}^{\theta,x}\left(B_{r}\left(x\right)\right)\right)\mathcal{W}\left(d\theta\right)dx.\end{split}
\label{eq:transit from M to W*N}
\end{equation}
 In particular, this implies that for every $r>0$, the joint distribution
of $\left(x,m^{\theta}\left(B_{r}\left(x\right)\right)\right)$ under
$\mathcal{M}\left(dxd\theta\right)$ is the same as the distribution
of $\left(x,\hat{m}^{\theta,x}\left(B_{r}\left(x\right)\right)\right)$
under $\mathcal{W}\left(d\theta\right)dx$, whose marginal distribution
on $\Theta$ at $x$ is independent of $x$. 
\end{lem}
Based on Lemma \ref{lem:area dist under W*N}, instead of $m^{\theta}\left(B_{r}\left(x\right)\right)$
under $\mathcal{M}\left(dxd\theta\right)$, we may as well study the
distribution of $\hat{m}^{\theta,x}\left(B_{r}\left(x\right)\right)$
under $\mathcal{W}\left(d\theta\right)dx$. Similarly, to understand
$r_{\Lambda}\left(x,\theta\right)$ under $\mathcal{M}\left(dxd\theta\right)$,
we only need to look at the random variable given by 
\begin{equation}
\left(x,\theta\right)\mapsto\hat{r}_{\Lambda}\left(x,\theta\right)\equiv\begin{cases}
\sup\left\{ r>0:\hat{m}^{\theta,x}\left(B_{r}\left(x\right)\right)\leq\Lambda\right\}  & \mbox{if }\theta\in\Theta_{x},\\
0 & \mbox{otherwise.}
\end{cases}\label{eq:def of r^under W*N}
\end{equation}
under $\mathcal{W}\left(d\theta\right)dx$, whose marginal distribution
on $\Theta$ at $x$ is again independent of $x$.

To proceed from here, we will follow the same strategy as in \cite{DS1}.
For the sake of completeness, we will still present the main steps
here. For every $r>0$ and $\Lambda>0$, since the distribution of
$\hat{m}^{\theta,x}\left(B_{r}\left(x\right)\right)$ and $\hat{r}_{\Lambda}\left(x,\theta\right)$
under $\mathcal{W}$ does not depend on $x$, we can assume $x$ is
the origin and simplify the notation by denoting $B_{r}\equiv B_{r}\left(0\right)$,
$\hat{r}_{\Lambda}\left(\theta\right)\equiv\hat{r}_{\Lambda}\left(0,\theta\right)$
and $\hat{m}^{\theta}\left(B_{r}\right)\equiv\hat{m}^{\theta,0}\left(B_{r}\left(0\right)\right)$.
We want to find an ``approximation'' for $\hat{m}^{\theta}\left(B_{r}\right)$
by conditioning on the value of the GFF restricted to the ``boundary''
of $B_{r}$. To be precise, recall from the definition (\ref{eq:def of X_t})
that if $r\left(t\right)\equiv G^{-1}\left(t+G\left(R\right)\right)$
with $t\geq0$, then $X_{t}=\mathcal{I}\left(h_{\mu_{r\left(t\right)}^{0}}\right)-\mathcal{I}\left(h_{\mu_{R}^{0}}\right)$
has the same distribution of a standard Brownian motion. We want to
investigate the conditional expectation of $\hat{m}^{\theta}\left(B_{r\left(t\right)}\right)$
given $X_{t}$. To do this, we need to relate $\hat{m}^{\theta}\left(B_{r\left(t\right)}\right)$
to the approximating measures $m_{\epsilon_{n}}^{\theta}\left(dx\right)$,
which requires us to overcome the singularity of $e^{\frac{\gamma^{2}}{2\pi^{2}}K_{0}\left(\left|\cdot\right|\right)}$
at the origin. To this end, assume $\left\{ f_{l}:l\geq1\right\} \subseteq C_{c}\left(\overline{B_{R}}\right)$
is a sequence with $0\leq f_{l}\nearrow\chi_{B_{r\left(t\right)}}$
as $l\rightarrow\infty$, then%
\footnote{The notation ``$\alpha\vee\beta$'' denotes ``$\max\left\{ \alpha,\beta\right\} $''
and ``$\alpha\wedge\beta$'' denotes ``$\min\left\{ \alpha,\beta\right\} $''.%
}

\[
C_{c}\left(\overline{B_{R}}\right)\ni d_{l}\left(\cdot\right)\equiv f_{l}\left(\cdot\right)e^{\frac{\gamma^{2}}{2\pi^{2}}\left(K_{0}\left(\left|\cdot\right|\right)\wedge l\right)}\nearrow\chi_{B_{r\left(t\right)}}\left(\cdot\right)e^{\frac{\gamma^{2}}{2\pi^{2}}K_{0}\left(\left|\cdot\right|\right)}.
\]
Therefore, one can apply the convergence results in Theorem \ref{thm: construction of the measure}
to $d_{l}$ for every $l\geq1$. Together with the monotone convergence
theorem, one sees that for every $t\geq0$, 
\begin{equation}
\mathbb{E}^{\mathcal{W}}\left[\hat{m}^{\theta}\left(B_{r\left(t\right)}\right)|X_{t}\right]=\lim_{l\rightarrow\infty}\liminf_{n\rightarrow\infty}\mathbb{E}^{\mathcal{W}}\left[M_{\epsilon_{n}}^{\theta}\left(d_{l}\right)|X_{t}\right].\label{eq:cond exp limit eq}
\end{equation}
Given the order of taking limits in the right hand side of (\ref{eq:cond exp limit eq}),
for every $l\geq1$ and eventually all large $n$, we have $\overline{B_{\epsilon_{n}}\left(y\right)}\subseteq B_{r\left(t\right)}\subseteq B_{R}$
for every $y\in\mbox{supp}\left(d_{l}\right)$. Thus by a simple exercise
on conditional expectations of Gaussian random variables along with
(\ref{eq:cov inclusion}), we can derive from (\ref{eq:cond exp limit eq})
that 
\begin{equation}
\begin{aligned} & \mathbb{E}^{\mathcal{W}}\left[\hat{m}^{\theta}\left(B_{r(t)}\right)|X_{t}\right]\\
=\quad & \int_{B_{r\left(t\right)}}e^{\frac{\gamma^{2}}{2\pi^{2}}K_{0}\left(\left|y\right|\right)}\exp\left[\left(I_{0}\left(\left|y\right|\right)-I_{2}\left(\left|y\right|\right)P\left(t\right)\right)\gamma X_{t}\right]\\
 & \qquad\qquad\qquad\qquad\qquad\quad\cdot\exp\left[-\frac{\gamma^{2}t}{2}\left(I_{0}\left(\left|y\right|\right)-I_{2}\left(\left|y\right|\right)P\left(t\right)\right)^{2}\right]dy,
\end{aligned}
\label{eq:cond exp of area}
\end{equation}
where $P\left(t\right)\equiv\left(4\pi^{2}t\right)^{-1}\left[\left(I_{1}^{2}-I_{0}I_{2}\right)^{-1}\left(r\left(t\right)\right)-\left(I_{1}^{2}-I_{0}I_{2}\right)^{-1}\left(R\right)\right]$.

If one carefully examines the asymptotics near the origin of the Bessel
functions involved, one realizes that (\ref{eq:cond exp of area})
suggests the conditional expectation of $\hat{m}^{\theta}\left(B_{r(t)}\right)$
given $X_{t}$, when $t$ is large, is ``approximately'' 
\begin{equation}
\hat{m}^{\theta*}\left(B_{r(t)}\right)\equiv\exp\left(\gamma X_{t}-\left(8\pi^{2}-\frac{\gamma^{2}}{2}\right)t\right).\label{eq:cond exp of area in terms of X_t}
\end{equation}
For the moment we will ``pretend'' $\hat{m}^{\theta}\left(B_{r(t)}\right)$
is just $\hat{m}^{\theta*}\left(B_{r\left(t\right)}\right)$ and formulate
the KPZ relation under this circumstance.

For every $\Lambda>0$, we define the stopping time: 
\begin{equation}
T_{\Lambda}^{*}\equiv\inf\left\{ t\geq0:\hat{m}^{\theta*}\left(B_{r\left(t\right)}\right)=\exp\left(\gamma X_{t}-\left(8\pi^{2}-\frac{\gamma^{2}}{2}\right)t\right)\leq\Lambda\right\} .\label{eq:def of T^*}
\end{equation}
The distribution of $T_{\Lambda}^{*}$ can be completely determined
by a martingale argument, which is straightforward but worth repeating.
Namely, for every $s\leq0$, by Doob's stopping time theorem, $\left\{ \exp\left[sX_{t\wedge T_{\Lambda}^{*}}-\frac{s^{2}}{2}\left(t\wedge T_{\Lambda}^{*}\right)\right]:t\geq0\right\} $
is a uniformly bounded martingale. Furthermore, the continuity of
Brownian motion implies that 
\[
X_{T_{\Lambda}^{*}}=\frac{\log\Lambda}{\gamma}+\left(8\pi^{2}-\frac{\gamma^{2}}{2}\right)\frac{T_{\Lambda}^{*}}{\gamma}.
\]
Therefore, the fact that the expectation of the martingale at $t=0$
is equal to that at $t=T_{\Lambda}^{*}$ leads to the formula of the
moment generating function of $T_{\Lambda}^{*}$: 
\begin{equation}
\mathbb{E}^{\mathcal{W}}\left[\exp\left(-\frac{\gamma s^{2}-2s\left(8\pi^{2}-\frac{\gamma^{2}}{2}\right)}{2\gamma}T_{\Lambda}^{*}\right)\right]=\Lambda^{-s/\gamma}.\label{eq:char of T-Delta}
\end{equation}
From here we can derive our first version of the KPZ relation which
is easy but revealing.
\begin{lem}
\label{lem:KPZ for m^*}Assume $D\subseteq\mathbb{R}^{4}$ is a bounded
Borel set with Euclidean scaling exponent $\kappa\in\left[0,1\right]$,
i.e., 
\begin{equation}
\lim_{\lambda\downarrow0}\frac{\log\textrm{Vol}\left(D_{\lambda}\right)}{\log\lambda^{4}}=\kappa.\label{eq:exponent-Euclid.}
\end{equation}
For every $\Lambda>0$, let $T_{\Lambda}^{*}$ be as in (\ref{eq:def of T^*})
and define the random ``radius'': 
\[
\theta\mapsto r_{\Lambda}^{*}\left(\theta\right)\equiv G^{-1}\left(T_{\Lambda}^{*}\left(\theta\right)+G\left(R\right)\right),
\]
and the random ``neighborhood'': 
\[
\theta\mapsto D^{\Lambda*,\theta}\equiv\cup_{x\in D}B_{r_{\Lambda}^{*}\left(\theta\right)}\left(x\right).
\]
Then, we have 
\begin{equation}
\lim_{\Lambda\downarrow0}\frac{\log\mathbb{E}^{\mathcal{W}}\left[m^{\theta}\left(D^{\Lambda*,\theta}\right)\right]}{\log\Lambda}=\lim_{\Lambda\downarrow0}\frac{\log\mathbb{E}^{\mathcal{W}}\left[\left(r_{\Lambda}^{*}\right)^{4\kappa}\right]}{\log\Lambda}=K\label{eq:exponent-m^theta}
\end{equation}
 where $K\in\left[0,1\right]$ is determined by the following quadratic
relation with $\kappa$: 
\begin{equation}
\kappa=K\left(1-\frac{\gamma^{2}}{16\pi^{2}}\right)+\frac{\gamma^{2}}{16\pi^{2}}K^{2}.\label{eq:KPZ relation for m^*}
\end{equation}
 \end{lem}
\begin{proof}
Clearly in this setting, we want to cover $D$ with open balls that
all have the same ``critical'' radius determined by the stopping
time associated with the martingale $\hat{m}^{\theta*}\left(B_{r\left(t\right)}\right)$
(defined in (\ref{eq:cond exp of area in terms of X_t})). We don't
have to worry about the measurability of $\theta\mapsto m^{\theta}\left(D^{\Lambda*,\theta}\right)$
because both $\theta\mapsto r_{\Lambda}^{*}\left(\theta\right)$ and
$\left(\theta,r\right)\mapsto m^{\theta}\left(D_{r}\right)$ are measurable.
Conditioning on $r_{\Lambda}^{*}=r\in(0,R]$, $D^{\Lambda*,\theta}$
is bounded and open, and the conditional expectation of $m^{\theta}\left(D^{\Lambda*,\theta}\right)$
is a multiple (reciprocal of the probability density function of $r_{\Lambda}^{*}$
at $r$) of $\mbox{vol}\left(D_{r}\right)$, which, according to (\ref{eq:exponent-Euclid.}),
\footnote{Throughout this article, the notation ``$\approx$'' means ``bounded
from above and below by a universal constant multiple of''.%
}``$\approx$'' $r^{4\kappa}$ for every $r\in(0,R]$. This means
$\mathbb{E}^{\mathcal{W}}\left[m^{\theta}\left(D^{\Lambda*,\theta}\right)\right]$
$\approx\mathbb{E}^{\mathcal{W}}\left[\left(r_{\Lambda}^{*}\right)^{4\kappa}\right]$
and further $\approx\mathbb{E}^{\mathcal{W}}\left[\exp\left(-8\pi^{2}\kappa T_{\Lambda}^{*}\right)\right]$.
Given (\ref{eq:char of T-Delta}), clearly one wants to set $8\pi^{2}\kappa$
to be $\frac{s^{2}}{2}-\frac{s}{\gamma}\left(8\pi^{2}-\frac{\gamma^{2}}{2}\right)$
for some $s\in\left[-\gamma,0\right]$, in which case 
\[
\lim_{\Lambda\downarrow0}\frac{\log\mathbb{E}^{\mathcal{W}}\left[m^{\theta}\left(D^{\Lambda*,\theta}\right)\right]}{\log\Lambda}=\lim_{\Lambda\downarrow0}\frac{\log\mathbb{E}^{\mathcal{W}}\left[\exp\left(-8\pi^{2}\kappa T_{\Lambda}^{*}\right)\right]}{\log\Lambda}=-\frac{s}{\gamma}.
\]
The results in (\ref{eq:exponent-m^theta}) and (\ref{eq:KPZ relation for m^*})
follow immediately after setting $K\equiv-\frac{s}{\gamma}$. 
\end{proof}
Next, we argue that $\hat{m}^{\theta*}\left(B_{r}\right)$ is indeed
a ``legitimate'' approximation for $\hat{m}^{\theta}\left(B_{r}\right)$
in the sense that $r_{\Lambda}^{*}$, as defined in Lemma \ref{lem:KPZ for m^*},
approximates $\hat{r}_{\Lambda}:\theta\mapsto\hat{r}_{\Lambda}\left(\theta\right)$
when ``compared'' in the limit of the logarithm ratio.
\begin{lem}
\label{lem: ratio of r^* with r}Assume the pair $\left(\kappa,K\right)\in\left[0,1\right]^{2}$
satisfies the quadratic relation in (\ref{eq:KPZ relation for m^*}).
Then,  
\begin{equation}
\lim_{\Lambda\downarrow0}\frac{\log\mathbb{E}^{\mathcal{W}}\left[\left(\hat{r}_{\Lambda}\right)^{4\kappa}\right]}{\log\Lambda}=K\mbox{ or equivalently }\lim_{\Lambda\downarrow0}\frac{\log\mathbb{E}^{\mathcal{W}}\left[\left(\hat{r}_{\Lambda}\right)^{4\kappa}\right]}{\log\mathbb{E}^{\mathcal{W}}\left[\left(r_{\Lambda}^{*}\right)^{4\kappa}\right]}=1.\label{eq:ratio of r to r^*}
\end{equation}

\end{lem}
The proof of this lemma is given in Section 5. There we also prove
a preliminary result (Lemma \ref{lem: tail estimate}) which actually
implies the almost sure non-triviality of the measure $\hat{m}^{\theta}\left(dx\right)$
as well as $m^{\theta}\left(dx\right)$. Most importantly, this lemma
builds up the final passage to the KPZ relation for $m^{\theta}\left(dx\right)$,
the ``true'' case in which we are interested. Again, we will only
present the statement here and leave the proof to the next section.
\begin{thm}
\label{thm:KPZ}Let $D\subseteq\mathbb{R}^{4}$ be a bounded Borel
set with Euclidean scaling exponent $\kappa\in\left[0,1\right]$.
Then $D$ has quantum scaling exponent $K\in\left[0,1\right]$ as
defined in (\ref{eq:def of K}), where $K$ is related to $\kappa$
by (\ref{eq:KPZ relation for m^*}).
\end{thm}

\section{Proofs of Results in Section 4}

We will now prove Lemma \ref{lem:area dist under W*N}. The strategy
is to relate $m^{\theta}\left(dx\right)$ to the approximating measures
$m_{\epsilon_{n}}^{\theta}\left(dx\right)$ and recognize that, by
the Cameron-Martin formula, the density of $m_{\epsilon_{n}}^{\theta}\left(dx\right)$,
i.e., $E_{\epsilon_{n}}^{\theta}\left(x\right)=\exp\left(\gamma\mathcal{I}\left(h_{\mu_{\epsilon_{n}}^{x}}\right)\left(\theta\right)-\frac{\gamma^{2}}{2}G\left(\epsilon_{n}\right)\right)$,
is just the Radon-Nikodym derivative with respect to $\mathcal{W}$
of the Gaussian measure induced by the translation $\theta\mapsto\theta+\gamma h_{\mu_{\epsilon_{n}}^{x}}$
under $\mathcal{W}$. Note that the constraint $0<\gamma^{2}<2\pi^{2}$
becomes necessary in this proof. \\
\emph{}\\
\emph{Proof of Lemma \ref{lem:area dist under W*N}:} Let $\hat{m}^{\theta,x}\left(dy\right)$
be the measure as defined in (\ref{eq:area dist under W*N}). The
claim that $\hat{m}^{\theta,x}\left(dy\right)$ is non-negative regular
and $\sigma-$finite follows from the observation that $\exp\left(\frac{\gamma^{2}}{2\pi^{2}}K_{0}\left(\left|x-\cdot\right|\right)\right)$
is locally integrable with respect to $m^{\theta}\left(dx\right)$
if $\theta\in\Theta_{x}$. Without loss of generality, we will assume
$x=0$. The only possible problem comes from the singularity at $0$.
However, if we rewrite
\[
\begin{split}\int_{B_{\epsilon_{0}}\left(0\right)}e^{\frac{\gamma^{2}}{2\pi^{2}}K_{0}\left(\left|y\right|\right)}m^{\theta}\left(dy\right) & =\sum_{k=0}^{\infty}\int_{\epsilon_{k}\leq\left|y\right|<\epsilon_{k-1}}e^{\frac{\gamma^{2}}{2\pi^{2}}K_{0}\left(\left|y\right|\right)}m^{\theta}\left(dy\right)\\
 & \leq\sum_{k=0}^{\infty}e^{\frac{\gamma^{2}}{2\pi^{2}}K_{0}\left(\epsilon_{k}\right)}m^{\theta}\left(B_{\epsilon_{k-1}}\left(0\right)\right),
\end{split}
\]
then the criterion (\ref{eq:def of Theta_x}) for $\Theta_{x}$ guarantees
that the series in the right hand side of above is convergent. 

Now we move on to the second part of the lemma. Clearly both mappings
\[
\left(x,\theta\right)\mapsto F\left(x,m^{\theta}\left(B_{r}\left(x\right)\right)\right)\mbox{ and }\left(x,\theta\right)\mapsto F\left(x,\hat{m}^{\theta,x}\left(B_{r}\left(x\right)\right)\right)
\]
are measurable with respect to $\mathfrak{B}_{\mathbb{R}^{4}}\times\mathfrak{B}_{\Theta}$,
so the two integrals in (\ref{eq:transit from M to W*N}) are well
defined and in fact finite. Choose a continuous mapping $x\in\Gamma\mapsto\rho^{x}\in C_{0}\left(\mathbb{R}^{4}\right)$
with $0\leq\rho^{x}<\chi_{B_{r}\left(x\right)}$. We first show (\ref{eq:transit from M to W*N})
holds with $\chi_{B_{r}\left(x\right)}$ replaced by $\rho^{x}$.
Namely, we claim that 
\begin{equation}
\begin{split}\int_{\Theta}\int_{\Gamma}F\left(x,M^{\theta}\left(\rho^{x}\right)\right)m^{\theta}\left(dx\right)\mathcal{W}\left(d\theta\right)=\qquad\qquad\qquad\qquad\qquad\qquad\\
\int_{\Gamma}\int_{\Theta}F\left(x,M^{\theta}\left(\rho^{x}e^{\frac{\gamma^{2}}{2\pi^{2}}K_{0}\left(\left|x-\cdot\right|\right)}\right)\right)\mathcal{W}\left(d\theta\right)dx.
\end{split}
\label{eq:transit from M to W*N with function rho}
\end{equation}

We start with rewriting the left hand side of (\ref{eq:transit from M to W*N with function rho}).
Since $F\left(x,M^{\theta}\left(\rho^{x}\right)\right)$ is continuous
in $x\in\Gamma$, the weak convergence result implies that 
\[
\int_{\Gamma}F\left(x,M^{\theta}\left(\rho^{x}\right)\right)m^{\theta}\left(dx\right)=\lim_{n\rightarrow\infty}\int_{\Gamma}F\left(x,M^{\theta}\left(\rho^{x}\right)\right)m_{\epsilon_{n}}^{\theta}\left(dx\right),
\]
which, by the dominated convergence theorem, leads to 
\begin{equation}
\begin{split}\int_{\Theta}\int_{\Gamma}F\left(x,M^{\theta}\left(\rho^{x}\right)\right)m^{\theta}\left(dx\right)\mathcal{W}\left(d\theta\right)=\qquad\qquad\qquad\qquad\qquad\\
\lim_{n\rightarrow\infty}\int_{\Theta}\int_{\Gamma}F\left(x,M^{\theta}\left(\rho^{x}\right)\right)E_{\epsilon_{n}}^{\theta}\left(x\right)dx\mathcal{W}\left(d\theta\right)
\end{split}
.\label{eq:1st step in area under M}
\end{equation}
By Fubini's Theorem and the consideration (about viewing $E_{\epsilon_{n}}^{\theta}\left(x\right)$
as the Radon-Nikodym derivative of the translated Wiener measure)
we made before the proof, we have that the right hand side of (\ref{eq:1st step in area under M})
equals 
\[
\lim_{n\rightarrow\infty}\int_{\Gamma}\int_{\Theta}F\left(x,M^{\theta+\gamma h_{\mu_{\epsilon_{n}}^{x}}}\left(\rho^{x}\right)\right)\mathcal{W}\left(d\theta\right)dx.
\]
Now given $x\in\mathbb{R}^{4}$, the Cameron-Martin theorem guarantees
that also with probability 1, $m_{\epsilon_{k}}^{\theta+\gamma h_{\mu_{\epsilon_{n}}^{x}}}\left(dy\right)$
weakly converges to $m^{\theta+\gamma h_{\mu_{\epsilon_{n}}^{x}}}\left(dy\right)$
as $k\rightarrow\infty$ simultaneously for all $n\ge1$. In particular,
\[
\begin{split}M^{\theta+\gamma h_{\mu_{\epsilon_{n}}^{x}}}\left(\rho^{x}\right) & =\lim_{k\rightarrow\infty}M_{\epsilon_{k}}^{\theta+\gamma h_{\mu_{\epsilon_{n}}^{x}}}\left(\rho^{x}\right)\\
 & =\lim_{k\rightarrow\infty}\int_{\mathbb{R}^{4}}\rho^{x}\left(y\right)\exp\left(\gamma^{2}\mathbb{E}\left[\mathcal{I}\left(h_{\mu_{\epsilon_{k}}^{y}}\right)\mathcal{I}\left(h_{\mu_{\epsilon_{n}}^{x}}\right)\right]\right)E_{\epsilon_{k}}^{\theta}\left(y\right)dy.
\end{split}
\]
At this point, it is clear that (\ref{eq:transit from M to W*N with function rho})
would follow if we can show that for every $\theta\in\Theta_{x}$,
\begin{equation}
\begin{split}\lim_{n\rightarrow\infty}\lim_{k\rightarrow\infty}\int_{\mathbb{R}^{4}}\rho^{x}\left(y\right)\exp\left(\gamma^{2}\mathbb{E}\left[\mathcal{I}\left(h_{\mu_{\epsilon_{k}}^{y}}\right)\mathcal{I}\left(h_{\mu_{\epsilon_{n}}^{x}}\right)\right]\right)E_{\epsilon_{k}}^{\theta}\left(y\right)dy\quad\;\qquad\qquad\\
=\int_{\mathbb{R}^{4}}\rho^{x}\left(y\right)\exp\left(\frac{\gamma^{2}}{2\pi^{2}}K_{0}\left(\left|x-y\right|\right)\right)m^{\theta}\left(dy\right)<\infty.
\end{split}
\label{eq:limit of three-piece integral}
\end{equation}
The right hand side of (\ref{eq:limit of three-piece integral}) is
finite for $\theta\in\Theta_{x}$ as we have seen in the proof of
the first part of this lemma. To establish the equation in (\ref{eq:limit of three-piece integral}),
we assume $n\geq1$ is sufficiently large and $k\geq n$ and divide
the integral in the left hand side of (\ref{eq:limit of three-piece integral})
into three pieces: 
\begin{equation}
\begin{split}\left\{ \int_{\left|y-x\right|<\epsilon_{n}-\epsilon_{k}}+\int_{\epsilon_{n}-\epsilon_{k}\leq\left|y-x\right|\leq\epsilon_{n}+\epsilon_{k}}+\int_{\left|y-x\right|>\epsilon_{n}+\epsilon_{k}}\right\} \qquad\qquad\qquad\qquad\qquad\qquad\\
\rho^{x}\left(y\right)\exp\left\{ \gamma\mathcal{I}\left(h_{\mu_{\epsilon_{k}}^{y}}\right)\left(\theta\right)+\gamma^{2}\mathbb{E}\left[\mathcal{I}\left(h_{\mu_{\epsilon_{k}}^{y}}\right)\mathcal{I}\left(h_{\mu_{\epsilon_{n}}^{x}}\right)\right]-\frac{\gamma^{2}}{2}G\left(\epsilon_{k}\right)\right\} dy.
\end{split}
\label{eq:chop up integral}
\end{equation}
We will investigate the limit as $k\rightarrow\infty$ and then $n\rightarrow\infty$
of each piece separately. 

By (\ref{eq:cov nonoverlap}), the last integral in (\ref{eq:chop up integral})
equals 
\begin{equation}
\int_{\left|y-x\right|>\epsilon_{n}+\epsilon_{k}}\rho^{x}\left(y\right)e^{\frac{\gamma^{2}}{2\pi^{2}}K_{0}\left(\left|x-y\right|\right)}E_{\epsilon_{k}}^{\theta}\left(y\right)dy.\label{eq: transition 1st integral}
\end{equation}
If the domain in (\ref{eq: transition 1st integral}) is replaced
by $\left\{ y:\,\left|y-x\right|>\epsilon_{n}\right\} $, then the
integral would be 
\begin{equation}
\begin{split}\int_{\left|y-x\right|>\epsilon_{n}}\rho^{x}\left(y\right)e^{\frac{\gamma^{2}}{2\pi^{2}}K_{0}\left(\left|x-y\right|\right)}E_{\epsilon_{k}}^{\theta}\left(y\right)dy=\qquad\qquad\qquad\qquad\qquad\qquad\qquad\\
\qquad\qquad M_{\epsilon_{k}}^{\theta}\left(\rho^{x}e^{\frac{\gamma^{2}}{2\pi^{2}}K_{0}\left(\left|x-\cdot\right|\vee\epsilon_{n}\right)}\right)-e^{\frac{\gamma^{2}}{2\pi^{2}}K_{0}\left(\epsilon_{n}\right)}\int_{\left|x-y\right|\leq\epsilon_{n}}\rho^{x}\left(y\right)E_{\epsilon_{k}}^{\theta}\left(y\right)dy.
\end{split}
\label{eq:transition 1st integral prime}
\end{equation}
As $k\rightarrow\infty$ and then $n\rightarrow\infty$, the first
term in the right hand side of (\ref{eq:transition 1st integral prime})
converges to $\int_{\mathbb{R}^{4}}\rho^{x}\left(y\right)e^{\frac{\gamma^{2}}{2\pi^{2}}K_{0}\left(\left|x-y\right|\right)}m^{\theta}\left(dy\right)$
(which is the term we want and also the only term that should survive
in the limit). On the other hand, as $k\rightarrow\infty$, the second
term on the right hand side of (\ref{eq:transition 1st integral prime})
is bounded by $e^{\frac{\gamma^{2}}{2\pi^{2}}K_{0}\left(\epsilon_{n}\right)}m^{\theta}\left(\overline{B_{\epsilon_{n}}\left(x\right)}\right)$,
which, because $\theta\in\Theta_{x}$, converges to zero when $n\rightarrow\infty$.
As for the ``redundant annulus'' which is the difference between
the left hand side of (\ref{eq:transition 1st integral prime}) and
(\ref{eq: transition 1st integral}), it's bounded by $e^{\frac{\gamma^{2}}{2\pi^{2}}K_{0}\left(\epsilon_{n}\right)}$
times the integral of $E_{\epsilon_{k}}^{\theta}$ over the annulus
$\left\{ \epsilon_{n}<\left|x-\cdot\right|\leq\epsilon_{n}+\epsilon_{k}\right\} $.
One can apply the Schwarz inequality to see that this integral is
bounded by 
\begin{equation}
\begin{split}e^{\frac{\gamma^{2}}{2}G\left(\epsilon_{k}\right)}\mbox{\mbox{vol}}\left(\left\{ \epsilon_{n}<\left|x-\cdot\right|\leq\epsilon_{n}+\epsilon_{k}\right\} \right)^{\frac{1}{2}}\left(\overline{m^{\theta,2\gamma}}\left(2\Gamma\right)\right)^{\frac{1}{2}}\end{split}
\label{eq:missing stripe in 1st integral}
\end{equation}
where $2\Gamma\equiv\left\{ 2y:\, y\in\Gamma\right\} $. Given the
considerations at the end of Section 2, without loss of generality,
we can assume $\overline{m^{\theta,2\gamma}}\left(2\Gamma\right)$
is finite and hence the limit of (\ref{eq:missing stripe in 1st integral})
as $k\rightarrow\infty$ with $n$ fixed is zero (because the volume
of the annulus $\approx\epsilon_{k}\approx e^{-2\pi^{2}G\left(\epsilon_{k}\right)}$
and $2\pi^{2}>\gamma^{2}$), so this ``annulus'' is negligible.

The second integral in (\ref{eq:chop up integral}) becomes negligible
by a similar argument. This time the Schwarz inequality implies that
the second integral is bounded by
\[
e^{\gamma^{2}\sqrt{G\left(\epsilon_{k}\right)G\left(\epsilon_{n}\right)}}e^{\frac{\gamma^{2}}{2}G\left(\epsilon_{k}\right)}\mbox{vol}^{\frac{1}{2}}\left(\left\{ \epsilon_{n}-\epsilon_{k}<\left|x-\cdot\right|<\epsilon_{n}+\epsilon_{k}\right\} \right)\left(\overline{m^{\theta,2\gamma}}\left(2\Gamma\right)\right)^{\frac{1}{2}}
\]
where we use the simple estimate $\mathbb{E}\left[\left|\mathcal{I}\left(h_{\mu_{\epsilon_{k}}^{y}}\right)\mathcal{I}\left(h_{\mu_{\epsilon_{n}}^{x}}\right)\right|\right]\leq\sqrt{G\left(\epsilon_{k}\right)G\left(\epsilon_{n}\right)}$.
Again, with $n$ fixed, the factor that involves $k$ converges to
zero. 

As for the first integral, because of (\ref{eq:cov inclusion}) and
the asymptotics of the Bessel functions involved, $\mathbb{E}\left[\mathcal{I}\left(h_{\mu_{\epsilon_{k}}^{y}}\right)\mathcal{I}\left(h_{\mu_{\epsilon_{n}}^{x}}\right)\right]$
is bounded by $\eta G\left(\epsilon_{n}\right)$ for some constant
$\eta\in\left(1,2\right)$ for all sufficiently large $n$. Therefore,
with $n$ fixed, the integral as $k\rightarrow\infty$ is bounded
by $e^{\eta\gamma^{2}G\left(\epsilon_{n}\right)}m^{\theta}\left(\overline{B_{\epsilon_{n}}\left(x\right)}\right)$,
and as $n\rightarrow\infty$ it therefore converges to zero.

So far we have proved the claim (\ref{eq:transit from M to W*N with function rho}).
To reach (\ref{eq:transit from M to W*N}), one takes a sequence $\left\{ \rho_{l}^{x}:\, l\geq1\right\} \subseteq C_{c}^{\infty}\left(\mathbb{R}^{4}\right)$
such that $0\leq\rho_{l}^{x}\nearrow\chi_{B_{r}\left(x\right)}$ as
$l\rightarrow\infty$, and for every $l\geq1$, $x\in\mathbb{R}^{4}\mapsto\rho_{l}^{x}\in C_{0}\left(\mathbb{R}^{4}\right)$
is continuous. So (\ref{eq:transit from M to W*N with function rho})
holds for each $l\geq1$. After carefully examining the integrals
on both sides of (\ref{eq:transit from M to W*N with function rho}),
one realizes that the limit as $l\rightarrow\infty$ can be passed
all the way inside to produce (\ref{eq:transit from M to W*N}). At
the end, it's clear that given $x$, the distribution of $\hat{m}^{\theta,x}\left(B_{r}\left(x\right)\right)$
under $\mathcal{W}$ is independent of $x$ due to the translation
invariance of measure $m^{\theta}\left(dy\right)$. Hence we have
completed the proof to Lemma \ref{lem:area dist under W*N}. $\square$ 

Now we move on to the proofs of the KPZ results. The techniques we
adopt here differ from those used in the two dimensional proofs in
\cite{DS1}, partly because of the absence in our setting of the two
dimensional conformal structure as well as the compactness of the
domain. For example, the next lemma is the ``tail estimate'' which
is the key estimate in proving both Lemma \ref{lem: ratio of r^* with r}
and Theorem \ref{thm:KPZ}. In the two dimensional counterpart, the
corresponding estimate (\cite{DS1}, §4.3) is a super-exponential
type of estimate. Below we prove an exponential type of estimate in
the four dimensional setting, but by carefully ``tuning'' the exponential
decay rate, we can still make it sufficient for our purposes. Again,
the occurrence of $\hat{m}^{\theta}\left(dy\right)$ in the next lemma
refers to the measure $e^{\frac{\gamma^{2}}{2\pi^{2}}K_{0}\left(\left|y\right|\right)}m^{\theta}\left(dy\right)$
assuming $\theta\in\Theta_{0}$. In other words, only balls centered
at the origin are concerned. However, since the distribution of $\hat{m}^{\theta,x}\left(B_{r}\left(x\right)\right)$
under $\mathcal{W}$ does not depend on $x$, the same result will
hold for $\hat{m}^{\theta,x}\left(dy\right)=e^{\frac{\gamma^{2}}{2\pi^{2}}K_{0}\left(\left|x-y\right|\right)}m^{\theta}\left(dy\right)$
(assuming $\theta\in\Theta_{x}$) no matter what $x$ is.\foreignlanguage{english}{ }
\selectlanguage{english}%
\begin{lem}
\label{lem: tail estimate}Let $B$ be the closed ball in $\mathbb{R}^{4}$
centered at the origin with unit volume under $e^{\frac{\gamma^{2}}{2\pi^{2}}K_{0}\left(\left|y\right|\right)}dy$,i.e.,
$\int_{B}e^{\frac{\gamma^{2}}{2\pi^{2}}K_{0}\left(\left|y\right|\right)}dy=1$.
If $\delta$ and $\rho$ are constants satisfying 
\[
0<\delta<4\pi^{2}-2\gamma^{2}\mbox{ and }\frac{4\pi^{2}+\gamma^{2}}{8\pi^{2}-\gamma^{2}-\delta}<\rho<1,
\]
then there exists $C>0$ such that%
\footnote{Throughout this section, $C$ denotes a constant that may depend on
$\gamma$, $\delta$, $\rho$ and $R$, but universal in $A$, $\epsilon_{n}$,
$x$ and $\Lambda$. The values of $C$ may change from line to line.%
} for all sufficiently large $A>0$, 
\begin{equation}
\mathcal{W}\left(\hat{m}^{\theta}\left(B\right)\leq e^{-A\gamma}\right)\leq C\exp\left[-\frac{2\rho}{\gamma}\left(8\pi^{2}-\gamma^{2}-\frac{\gamma^{2}}{\rho}-\delta\right)A\right].\label{eq:tail estimate}
\end{equation}
\end{lem}
\begin{proof}
Since $B$ is closed, it suffices to estimate $\mathcal{W}\left(\limsup_{n\rightarrow\infty}\hat{m}_{\epsilon_{n}}^{\theta}\left(B\right)\leq e^{-A\gamma}\right)$
where $\hat{m}_{\epsilon_{n}}^{\theta}\left(dy\right)$ has density
$e^{\frac{\gamma^{2}}{2\pi^{2}}K_{0}\left(\left|y\right|\right)}$
with respect to $m_{\epsilon_{n}}^{\theta}\left(dy\right)$. By the
same argument as used in deriving the estimate (\ref{eq:total mass estimate of 2nd moment of difference}),
we can show that there exists constant $C>0$ such that for all $n\geq1$,
\[
\mathbb{E}^{\mathcal{W}}\left[\left|\hat{m}_{\epsilon_{n+1}}^{\theta}\left(B\right)-\hat{m}_{\epsilon_{n}}^{\theta}\left(B\right)\right|^{2}\right]\leq Ce^{-\left(8\pi^{2}-\gamma^{2}\right)G\left(\epsilon_{n}\right)}.
\]
\foreignlanguage{american}{For any $\delta$ with $0<\delta<4\pi^{2}-2\gamma^{2}$},
denote $\mathcal{A}_{n}^{\prime}$, $n\geq1$, the measurable set
\[
\left\{ \forall l\geq n,\left|\hat{m}_{\epsilon_{l+1}}^{\theta}\left(B\right)-\hat{m}_{\epsilon_{l}}^{\theta}\left(B\right)\right|\leq e^{-A\gamma}e^{-\frac{\delta}{2}G\left(\epsilon_{l}\right)}\right\} .
\]
Then it follows easily from \foreignlanguage{american}{Chebyshev's
inequality and the Borel-Cantelli Lemma that} $\mathcal{W}\left(\bigcup_{n=1}^{\infty}\mathcal{A}_{n}^{\prime}\right)=1$.
Moreover, if $\mathcal{A}_{1}=\mathcal{A}_{1}^{\prime}$ and $\mathcal{A}_{n}=\mathcal{A}_{n}^{\prime}\backslash\mathcal{A}_{n-1}^{\prime}$
for $n\ge2$, then there exists constant $C>0$ such that for all
$n\geq2$, 
\begin{equation}
\mathcal{W}\left(\mathcal{A}_{n}\right)\leq Ce^{2A\gamma}e^{-\left(8\pi^{2}-\gamma^{2}-\delta\right)G\left(\epsilon_{n}\right)}.\label{eq:esti for prob(A)}
\end{equation}

Set $\mathcal{B}\equiv\left\{ \limsup_{n\rightarrow\infty}\hat{m}_{\epsilon_{n}}^{\theta}\left(B\right)\leq e^{-A\gamma}\right\} $,
then $\mathcal{W}\left(\mathcal{B}\right)=\sum_{n=1}^{\infty}\mathcal{W}\left(\mathcal{B}\cap\mathcal{A}_{n}\right)$
and it's clear that $\theta\in\mathcal{B}\cap\mathcal{A}_{n}$ implies
$\hat{m}_{\epsilon_{n}}^{\theta}\left(B\right)\leq c_{\delta}e^{-A\gamma}$
where $c_{\delta}=1+\sum_{n=1}^{\infty}e^{-\frac{\delta}{2}G\left(\epsilon_{n}\right)}$.
Given any $\rho$ such that $\frac{4\pi^{2}+\gamma^{2}}{8\pi^{2}-\gamma^{2}-\delta}<\rho<1$
(notice that such $\rho$ always exists since $0<\delta<4\pi^{2}-2\gamma^{2}$
and $4\pi^{2}+\gamma^{2}<8\pi^{2}-\gamma^{2}-\delta$), we set up
the ``threshold'' $N\in\mathbb{N}$ which is the unique (recall
that $G$ is strictly decreasing on $\left(0,\infty\right)$) integer
such that 
\begin{equation}
G\left(\epsilon_{N}\right)<\frac{2\rho A}{\gamma}\mbox{ but }G\left(\epsilon_{N+1}\right)\geq\frac{2\rho A}{\gamma}.\label{eq:def for thresholds}
\end{equation}
\foreignlanguage{american}{The desired estimate (\ref{eq:tail estimate})
is trivial when $n\geq N+1$, because (\ref{eq:esti for prob(A)})
and (\ref{eq:def for thresholds}) implies 
\begin{eqnarray*}
\sum_{n=N+1}^{\infty}\mathcal{W}\left(\mathcal{B}\cap\mathcal{A}_{n}\right) & \leq & Ce^{2A\gamma}e^{-\left(8\pi^{2}-\gamma^{2}-\delta\right)G\left(\epsilon_{N+1}\right)}\\
 & \leq & C\exp\left[-\frac{2\rho}{\gamma}\left(8\pi^{2}-\gamma^{2}-\frac{\gamma^{2}}{\rho}-\delta\right)A\right].
\end{eqnarray*}
}When $n=1,\cdots,N$, we apply Jensen's inequality to see that
\begin{equation}
\begin{split} & \mathcal{W}\left(\hat{m}_{\epsilon_{n}}^{\theta}\left(B\right)\leq c_{\delta}e^{-A\gamma}\right)\\
\leq\; & \mathcal{W}\left(\exp\left[\int_{B}\left(\gamma\mathcal{I}\left(h_{\mu_{\epsilon_{n}}^{y}}\right)\left(\theta\right)-\frac{\gamma^{2}}{2}G\left(\epsilon_{n}\right)\right)e^{\frac{\gamma^{2}}{2\pi^{2}}K_{0}\left(\left|y\right|\right)}dy\right]\leq c_{\delta}e^{-A\gamma}\right)\\
\leq\; & \mathcal{W}\left(\int_{B}\mathcal{I}\left(h_{\mu_{\epsilon_{n}}^{y}}\right)\left(\theta\right)e^{\frac{\gamma^{2}}{2\pi^{2}}K_{0}\left(\left|y\right|\right)}dy\leq-A+\frac{\gamma}{2}G\left(\epsilon_{n}\right)+\frac{\log c_{\delta}}{\gamma}\right).
\end{split}
\label{eq:esti with Jensen}
\end{equation}
By Corollary \ref{cor:a.e. continuity in x}, without loss of generality,
we can assume that for all $n\geq1$ and every $\theta$, the function
$y\in B\mapsto\mathcal{I}\left(h_{\mu_{\epsilon_{n}}^{y}}\right)\left(\theta\right)$
is continuous and hence uniformly continuous on $B$. Therefore, one
can easily check (for example, by writing the integral as the limit
of a discrete sum of Gaussian random variables) that 
\[
\theta\in\Theta\mapsto\int_{B}\mathcal{I}\left(h_{\mu_{\epsilon_{n}}^{y}}\right)\left(\theta\right)e^{\frac{\gamma^{2}}{2\pi^{2}}K_{0}\left(\left|y\right|\right)}dy\in\mathbb{R}
\]
is also a centered Gaussian random variable for every $n\geq1$, and
furthermore, the variance can be bounded by a constant $M$ that is
universal in $n\geq1$. In fact, $M$ can be taken as a constant multiple
of 
\[
\iint_{B\times B}K_{0}\left(\left|x-y\right|\right)e^{\frac{\gamma^{2}}{2\pi^{2}}\left(K_{0}\left(\left|x\right|\right)+K_{0}\left(\left|y\right|\right)\right)}dxdy.
\]
 Since $G\left(\epsilon_{n}\right)<\frac{2\rho A}{\gamma}$ for $n=1,\cdots,N$,
$A-\frac{\gamma}{2}G\left(\epsilon_{n}\right)>(1-\rho)A$, and (\ref{eq:esti with Jensen})
implies that when $A$ is sufficiently large, 
\[
\begin{split}\mathcal{W}\left(\hat{m}_{\epsilon_{n}}^{\theta}\left(B\right)\leq c_{\delta}e^{-A\gamma}\right) & \leq\exp\left[-\frac{1}{2M}\left(A-\frac{\gamma}{2}G\left(\epsilon_{n}\right)-\frac{1}{\gamma}\log c_{\delta}\right)^{2}\right]\\
 & \leq\exp\left[-\frac{1}{2M}\left(\left(1-\rho\right)A-\frac{1}{\gamma}\log c_{\delta}\right)^{2}\right]\\
 & \leq\exp\left[-\frac{1}{4M}\left(1-\rho\right)^{2}A^{2}\right].
\end{split}
\]
In addition, (\ref{eq:def for thresholds}) implies that $N$ is approximately
a constant multiple of $A$. Therefore, when $A$ is large, 
\[
\sum_{n=1}^{N}\mathcal{W}\left(\mathcal{B}\cap\mathcal{A}_{n}\right)\leq\sum_{n=1}^{N}\mathcal{W}\left(\hat{m}_{\epsilon_{n}}^{\theta}\left(B\right)\leq c_{\delta}e^{-A\gamma}\right)\leq CAe^{-\frac{\left(1-\rho\right)^{2}A^{2}}{4M}}.
\]
So $\sum_{n=1}^{N}\mathcal{W}\left(\mathcal{B}\cap\mathcal{A}_{n}\right)$
actually decays super-exponentially fast as $A\rightarrow\infty$,
and this estimate can certainly be transformed into the desired form
as in (\ref{eq:tail estimate}).\foreignlanguage{american}{ }
\end{proof}
\selectlanguage{american}%
We are now ready to prove Lemma \ref{lem: ratio of r^* with r}. Recall
the notation $\hat{r}_{\Lambda}:\theta\mapsto\hat{r}_{\Lambda}\left(0,\theta\right)$
where $\hat{r}_{\Lambda}\left(0,\theta\right)$ is as defined in (\ref{eq:def of r^under W*N})
with $x$ being the origin. Let $\left(\kappa,K\right)\in\left[0,1\right]^{2}$
be a pair as in (\ref{eq:KPZ relation for m^*}). In order to get
(\ref{eq:ratio of r to r^*}), it suffices to show that 
\begin{equation}
C^{-1}\leq\Lambda^{-K}\mathbb{E}^{\mathcal{W}}\left[\left(\hat{r}_{\Lambda}\right)^{4\kappa}\right]\leq C\label{eq:upper and lower bound of log ratio}
\end{equation}
for some constant $C>0$ universal in $\Lambda$ as $\Lambda\downarrow0$.
We will prove the existence of the upper bound and the lower bound
in (\ref{eq:upper and lower bound of log ratio}) separately.\\
\\
\emph{Proof of the upper bound in (\ref{eq:upper and lower bound of log ratio}):}
For notational convenience, we introduce the ``stopping time'' corresponding
to $\hat{r}_{\Lambda}$, i.e., $T_{\Lambda}\equiv G\left(\hat{r}_{\Lambda}\right)-G\left(R\right)$.
We want to show $\Lambda^{-K}\mathbb{E}^{\mathcal{W}}\left[\exp\left(-8\pi^{2}\kappa T_{\Lambda}\right)\right]$
is bounded from above uniformly in small $\Lambda$. It's clear, from
(\ref{eq:tail estimate}) and the fact that $\mathcal{W}\left(\Theta_{0}\right)=1$,
where $\Theta_{0}$ is as in (\ref{eq:def of Theta_x}), that $T_{\Lambda}\in\left(-G\left(R\right),\infty\right)$
almost surely. Let constant $\delta$ and $\rho$ be as in the statement
of Lemma \ref{lem: tail estimate}. Set 
\[
S\equiv\frac{-\log\Lambda}{8\pi^{2}-\gamma^{2}}\frac{2\rho\left(8\pi^{2}-\gamma^{2}-\frac{\gamma^{2}}{\rho}-\delta\right)-K\gamma^{2}}{2\rho\left(8\pi^{2}-\gamma^{2}-\frac{\gamma^{2}}{\rho}-\delta\right)}.
\]
Then the expectation of $\exp\left(-8\pi^{2}\kappa T_{\Lambda}\right)$
can be written as 
\[
\mathbb{E}^{\mathcal{W}}\left[\exp\left(-8\pi^{2}\kappa T_{\Lambda}\right)\chi_{\left\{ -G\left(R\right)<T_{\Lambda}<S\right\} }\right]+\mathbb{E}^{\mathcal{W}}\left[\exp\left(-8\pi^{2}\kappa T_{\Lambda}\right)\chi_{\left\{ S\leq T_{\Lambda}<\infty\right\} }\right].
\]
In the first term, $T_{\Lambda}<S$ implies that the volume of the
closed ball centered at the origin with radius $r\left(S\right)=G^{-1}\left(S+G\left(R\right)\right)$
( $\approx\exp\left(-2\pi^{2}S\right)$) is no greater than $\Lambda$
under the measure $\hat{m}^{\theta}\left(dy\right)$, while this ball
has volume $\approx\left(r\left(S\right)\right)^{4-\frac{\gamma^{2}}{2\pi^{2}}}$
($\approx\exp\left(-\left(8\pi^{2}-\gamma^{2}\right)S\right))$ under
the measure $e^{\frac{\gamma^{2}}{2\pi^{2}}K_{0}\left(\left|y\right|\right)}dy.$
However, by Lemma \ref{lem: tail estimate}, the probability of this
event is bounded by 
\begin{eqnarray*}
 &  & C\exp\left(-\frac{2\rho}{\gamma}\left(8\pi^{2}-\gamma^{2}-\frac{\gamma^{2}}{\rho}-\delta\right)\left(\frac{-\log\Lambda}{\gamma}-\frac{8\pi^{2}-\gamma^{2}}{\gamma}S\right)\right),
\end{eqnarray*}
which, given this particular choice of $S$, is equal to a constant
multiple of $\Lambda^{K}.$ Therefore, the first piece of integral
causes no trouble.

The second integral is bounded by $e^{-8\pi^{2}\kappa S}$. Hence
we only need to check that $\Lambda^{-K}e^{-8\pi^{2}\kappa S}$, or
equivalently, $\exp\left(-K\log\Lambda-8\pi^{2}\kappa S\right)$ stays
bounded as $\Lambda\downarrow0$. In fact, we will show that for all
possible values of $\left(\kappa,K\right)$ and all sufficiently small
$\Lambda>0$, $K\log\Lambda+8\pi^{2}\kappa S\geq0$, that is (assuming
$\log\Lambda<0$), 
\begin{equation}
K\leq\frac{8\pi^{2}\kappa}{8\pi^{2}-\gamma^{2}}\frac{2\rho\left(8\pi^{2}-\gamma^{2}-\frac{\gamma^{2}}{\rho}-\delta\right)-K\gamma^{2}}{2\rho\left(8\pi^{2}-\gamma^{2}-\frac{\gamma^{2}}{\rho}-\delta\right)}.\label{eq:sign determination on K and kappa}
\end{equation}
To simplify the notations, let's write $\zeta\equiv2\rho\left(8\pi^{2}-\gamma^{2}-\frac{\gamma^{2}}{\rho}-\delta\right)$.
Recall from the statement of Lemma \ref{lem: tail estimate} that
$0<\delta<4\pi^{2}-2\gamma^{2}$ and $\frac{4\pi^{2}+\gamma^{2}}{8\pi^{2}-\gamma^{2}-\delta}<\rho<1$,
so $\zeta>8\pi^{2}$. If we express $\kappa$ in terms of $K$ according
to (\ref{eq:KPZ relation for m^*}), then the statement in (\ref{eq:sign determination on K and kappa})
is equivalent to
\[
F\left(K\right)\equiv\gamma^{2}K^{2}+\left(16\pi^{2}-\gamma^{2}-\zeta\right)K-\zeta\leq0
\]
for all possible values of $K\in\left[0,1\right]$. However, this
is clearly true since $F$ is quadratic and $F\left(0\right)=-\zeta<0$
as well as $F\left(1\right)=16\pi^{2}-2\zeta<0$. $\square$\\
\\
\emph{Proof of the lower bound in (\ref{eq:upper and lower bound of log ratio}):}
Recall that $T_{\Lambda}^{*}$ is the stopping time (as defined in
(\ref{eq:def of T^*})) associated with the ``approximating'' measure
$\hat{m}^{\theta*}$, and 
\[
\mathbb{E}^{\mathcal{W}}\left[\exp\left(-8\pi^{2}\kappa T_{\Lambda}^{*}\right)\right]=\Lambda^{K}.
\]
We observe that $\mathbb{E}^{\mathcal{W}}\left[\exp\left(-8\pi^{2}\kappa T_{\Lambda}\right)\right]$
is greater than the integral of $\exp\left(-8\pi^{2}\kappa T_{\Lambda}\right)$
over the subset $\left\{ T_{\Lambda}\leq T_{\Lambda}^{*}\right\} $,
where the integrand is greater or equal to $\exp\left(-8\pi^{2}\kappa T_{\Lambda}^{*}\right)$.
Therefore, we have 
\[
\mathbb{E}^{\mathcal{W}}\left[\exp\left(-8\pi^{2}\kappa T_{\Lambda}\right)\right]\geq\mathbb{E}^{\mathcal{W}}\left[\exp\left(-8\pi^{2}\kappa T_{\Lambda}^{*}\right)\right]-\mathbb{E}^{\mathcal{W}}\left[\exp\left(-8\pi^{2}\kappa T_{\Lambda}^{*}\right)\chi_{\left\{ T_{\Lambda}>T_{\Lambda}^{*}\right\} }\right].
\]
It is clear that in order to get the desired lower bound, we need
to find constant $0<c<1$ such that 
\begin{equation}
\Lambda^{-K}\mathbb{E}^{\mathcal{W}}\left[\exp\left(-8\pi^{2}\kappa T_{\Lambda}^{*}\right)\chi_{\left\{ T_{\Lambda}>T_{\Lambda}^{*}\right\} }\right]\leq c\label{eq:lower bound ineq}
\end{equation}
uniformly in small $\Lambda$. Conditioning on $T_{\Lambda}^{*}=T$,
$T_{\Lambda}>T$ implies $\hat{m}^{\theta}\left(B_{r\left(T\right)}\right)>\Lambda$
and hence $\frac{\hat{m}^{\theta}\left(B_{r\left(T\right)}\right)}{\hat{m}^{\theta*}\left(B_{r\left(T\right)}\right)}>1$.
By Chebyshev's inequality, other than a factor given by the probability
density function of $T_{\Lambda}^{*}$, the conditional probability
of $\left\{ T_{\Lambda}>T\right\} $ is bounded by the expectation
of $\frac{\hat{m}^{\theta}\left(B_{r\left(T\right)}\right)}{\hat{m}^{\theta*}\left(B_{r\left(T\right)}\right)}$,
which, given the expression in (\ref{eq:cond exp of area}) (which
is the conditional expectation of the numerator given the denominator),
can be bounded by constant $c\in\left(0,1\right)$ which is universal
in $\Lambda$ and $T$. So the estimate in (\ref{eq:lower bound ineq})
will be satisfied by this choice of $c$. $\square$ 

$\quad$\emph{}\\
\emph{Proof of Theorem \ref{thm:KPZ}:} Assume $D\subseteq\overline{B_{N}\left(0\right)}$
for some sufficently large $N\geq1$. Let $r_{\Lambda}\left(x,\theta\right)$
and $D^{\Lambda,\theta}$ be as defined in (\ref{eq:def of r_Delta under M})
and (\ref{eq:vol-delta nbhd of D}). Denote $\mathcal{N}\left(d\theta dx\right)\equiv\mathcal{W}\left(d\theta\right)dx$.
Based on Lemma \ref{lem:area dist under W*N}, $\mathbb{E}^{\mathcal{W}}\left[m^{\theta}\left(D^{\Lambda,\theta}\right)\right]$
equals 
\begin{equation}
\begin{split} & \mathcal{M}\left(\left\{ \left(x,\theta\right):\mbox{ either }x\in D\mbox{ or dist}\left(x,D\right)<r_{\Lambda}\left(x,\theta\right)\right\} \right)\\
=\quad & \mathcal{N}\left(\left\{ \left(x,\theta\right):\left|x\right|\leq2N,\mbox{dist}\left(x,D\right)<\hat{r}_{\Lambda}\left(x,\theta\right)\right\} \right)\\
 & \qquad+\lim_{2N\leq M\rightarrow\infty}\mathcal{N}\left(\left\{ \left(x,\theta\right):2N\leq\left|x\right|\leq M,\mbox{ dist}\left(x,D\right)<\hat{r}_{\Lambda}\left(x,\theta\right)\right\} \right).
\end{split}
\label{eq:split M-probability into x cpt and x uncpt}
\end{equation}
In the first term of the right hand side of (\ref{eq:split M-probability into x cpt and x uncpt}),
conditioning on $\hat{r}_{\Lambda}\left(x,\theta\right)$ under $\mathcal{N}\left(d\theta dx\right)$,
since its marginal distribution on $\Theta$ does not depend on $x$,
the conditional probability of the set is proportional to $\mbox{vol}\left(D_{\hat{r}_{\Lambda}\left(x,\theta\right)}\cap\overline{B_{2N}\left(0\right)}\right)$.
We further split the set into two cases: $\hat{r}_{\Lambda}\left(x,\theta\right)>N$
and $\hat{r}_{\Lambda}\left(x,\theta\right)\leq N$, the later of
which also implies $D_{\hat{r}_{\Lambda}\left(x,\theta\right)}\subseteq\overline{B_{2N}\left(0\right)}$.
Therefore, the first term can be rewritten as (up to a constant depending
on $N$) 
\[
\begin{split} & \mathbb{E}^{\mathcal{W}}\left[\mbox{vol}\left(D_{\hat{r}_{\Lambda}\left(x,\theta\right)}\right)\chi_{\left\{ \hat{r}_{\Lambda}\left(x,\theta\right)\leq N\right\} }\right]+\mathbb{E}^{\mathcal{W}}\left[\mbox{vol}\left(D_{\hat{r}_{\Lambda}\left(x,\theta\right)}\cap\overline{B_{2N}\left(0\right)}\right)\chi_{\left\{ \hat{r}_{\Lambda}\left(x,\theta\right)>N\right\} }\right]\\
= & \mathbb{E}^{\mathcal{W}}\left[\mbox{vol}\left(D_{\hat{r}_{\Lambda}\left(x,\theta\right)}\right)\right]-\mathbb{E}^{\mathcal{W}}\left[\mbox{vol}\left(D_{\hat{r}_{\Lambda}\left(x,\theta\right)}\right)\chi_{\left\{ \hat{r}_{\Lambda}\left(x,\theta\right)>N\right\} }\right]\\
 & \qquad\qquad\qquad\qquad\qquad\qquad\qquad+\mathbb{E}^{\mathcal{W}}\left[\mbox{vol}\left(D_{\hat{r}_{\Lambda}\left(x,\theta\right)}\cap\overline{B_{2N}\left(0\right)}\right)\chi_{\left\{ \hat{r}_{\Lambda}\left(x,\theta\right)>N\right\} }\right].
\end{split}
\]
According to the assumption (\ref{eq:exponent-Euclid.}) and Lemma
\ref{lem: ratio of r^* with r}, $\mathbb{E}^{\mathcal{W}}\left[\mbox{vol}\left(D_{\hat{r}_{\Lambda}\left(x,\theta\right)}\right)\right]\approx\Lambda^{K}$
when $\Lambda$ is sufficiently small. On the other hand, given $\hat{r}_{\Lambda}\left(x,\theta\right)>N$,
$D_{\hat{r}_{\Lambda}\left(x,\theta\right)}$ is always contained
in the ball centered at the origin with radius $2\hat{r}_{\Lambda}\left(x,\theta\right)$,
so the last two terms in the right hand side of the equation above
are both bounded by (up to a constant) 
\begin{equation}
\mathbb{E}^{\mathcal{W}}\left[\left(\hat{r}_{\Lambda}\left(x,\theta\right)\right)^{4}\chi_{\left\{ \hat{r}_{\Lambda}\left(x,\theta\right)>N\right\} }\right]\leq4\int_{[1,\infty)}u^{3}\mathcal{W}\left(\hat{r}_{\Lambda}\left(x,\theta\right)>u\right)du.\label{eq:KPZ tail bound 1}
\end{equation}
If $\zeta\equiv2\rho\left(8\pi^{2}-\gamma^{2}-\frac{\gamma^{2}}{\rho}-\delta\right)$
where $\delta$ and $\rho$ are the same as in the statement of Lemma
\ref{lem: tail estimate}, then by (\ref{eq:tail estimate}), \foreignlanguage{english}{
\[
\mathcal{W}\left(\hat{r}_{\Lambda}\left(x,\theta\right)>u\right)\leq\mathcal{W}\left(\hat{m}^{\theta,x}\left(B_{u}\left(x\right)\right)\leq\Lambda\right)\leq C\Lambda^{\frac{\zeta}{\gamma^{2}}}u^{-\frac{\zeta}{\gamma^{2}}\left(4-\frac{\gamma^{2}}{2\pi^{2}}\right)}.
\]
}Given the particular range of $\delta,\rho$ and $\zeta$, one sees
that not only is the integral in (\ref{eq:KPZ tail bound 1}) finite,
but it also converges to zero faster than $\Lambda^{K}$ as $\Lambda\downarrow0$
for any possible value of $K\in[0,1]$. 

In the second term in (\ref{eq:split M-probability into x cpt and x uncpt}),
since $D\subseteq\overline{B_{N}\left(0\right)}$, the assumptions
$\left|x\right|\geq2N$ and $\mbox{dist}\left(x,D\right)<\hat{r}_{\Lambda}\left(x,\theta\right)$
imply $\hat{r}_{\Lambda}\left(x,\theta\right)>\frac{1}{2}\left|x\right|$
whose probability, as we have seen earlier, is bounded by $C\Lambda^{\frac{\zeta}{\gamma^{2}}}\left|x\right|^{-\frac{\zeta}{\gamma^{2}}\left(4-\frac{\gamma^{2}}{2\pi^{2}}\right)}$
which is integrable (with respect to $dx$) in the entire domain $\left\{ \left|x\right|\geq2N\right\} $.
Therefore the second term also converges to zero faster than $\Lambda^{K}$
as $\Lambda\downarrow0$. To summarize, we have shown that $\mathbb{E}^{\mathcal{W}}\left[m^{\theta}\left(D^{\Lambda,\theta}\right)\right]$
is a constant multiple of $\Lambda^{K}+o\left(\Lambda^{K}\right)$
as $\Lambda\downarrow0$ which is sufficient for the desired conclusion.
$\square$

\section{Possible Generalizations }

\subsection*{Generalizations to $\mathbf{\mathbb{R}^{2n}}$:}

In this subsection, we outline a possible generalization of the four
dimensional treatments carried out in previous sections to higher
even dimensions $\mathbb{R}^{2n}$ with $n\geq2$. We consider the
GFF on $\mathbb{R}^{2n}$ with the underlying Hilbert space $H\equiv H^{n}\left(\mathbb{R}^{2n}\right)$
which is the completion of the Schwartz test function space $\mathcal{S}\left(\mathbb{R}^{2n}\right)$
under the inner product $\left(\left(I-\Delta\right)^{n}\cdot,\cdot\right)_{L^{2}}$.
Similarly, for every $x\in\mathbb{R}^{2n}$ and $\epsilon>0$, $\sigma_{\epsilon}^{x}$
denotes the tempered distribution which is to take the spherical average
of a test function over the sphere $S_{\epsilon}\left(x\right)$.
In this setting, $\sigma_{\epsilon}^{x}\in H^{-n}\left(\mathbb{R}^{2n}\right)$
and again if $h_{\sigma_{\epsilon}^{x}}\equiv\left(I-\Delta\right)^{-n}\sigma_{\epsilon}^{x}$,
then $h_{\sigma_{\epsilon}^{x}}\in H$ and the Paley-Wiener integral
$\mathcal{I}\left(h_{\sigma_{\epsilon}^{x}}\right)$ can be viewed
as the ``generalized'' action of $\sigma_{\epsilon}^{x}$ on the
GFF. Moreover, the higher order of the operator $\left(I-\Delta\right)^{n}$
allows us to take higher ``derivatives'' of $\sigma_{\epsilon}^{x}$
in the radial variable $\epsilon$. If $d^{m}\sigma_{\epsilon}^{x}\equiv\frac{d^{m}}{d\epsilon^{m}}\sigma_{\epsilon}^{x}$
is defined in the sense of tempered distribution for every $m\in\mathbb{N}$,
then simple computations of the Fourier transforms show that 
\[
\hat{\sigma_{\epsilon}^{x}}\left(\xi\right)=C_{n}e^{i\left(x,\xi\right)_{\mathbb{R}^{2n}}}\left(\epsilon\left|\xi\right|\right)^{1-n}J_{n-1}\left(\epsilon\left|\xi\right|\right)
\]
 
\[
\mbox{ and }\left(d^{m}\sigma_{\epsilon}^{x}\right)^{\mathcircumflex}\left(\xi\right)=\frac{d^{m}}{d\epsilon^{m}}\hat{\sigma_{\epsilon}^{x}}\left(\xi\right)=C_{n}e^{i\left(x,\xi\right)_{\mathbb{R}^{2n}}}\frac{d^{m}}{d\epsilon^{m}}\left(\frac{J_{n-1}\left(\epsilon\left|\xi\right|\right)}{\left(\epsilon\left|\xi\right|\right)^{n-1}}\right),
\]
where $C_{n}>0$ is a dimensional constant. In particular, we can
write 
\[
\left(d^{m}\sigma_{\epsilon}^{x}\right)^{\mathcircumflex}\left(\xi\right)=C_{n}e^{i\left(x,\xi\right)_{\mathbb{R}^{2n}}}\varphi^{\left(m\right)}\left(\epsilon\left|\xi\right|\right)\left|\xi\right|^{m}
\]
where (\cite{BesselFunctions}, §3.31) $\varphi\left(r\right)\equiv\frac{J_{n-1}\left(r\right)}{r^{n-1}}$
for $r>0$ and $\varphi^{\left(m\right)}\left(r\right)$ is analytic
in $r$ near 0 and asymptotic to $r^{-\left(n-\frac{1}{2}\right)}$
as $r\rightarrow\infty$ for every $m\in\mathbb{N}$. Therefore, $\sigma_{\epsilon}^{x}$
and $d^{m}\sigma_{\epsilon}^{x}$ for $1\leq m\leq n-1$ are in $H^{-n}\left(\mathbb{R}^{2n}\right)$.

We can mimic the approach in Section 2 and define the vector-valued
Gaussian random variable on $\Theta$: for every $x\in\mathbb{R}^{2n}$
and $\epsilon>0$, 
\[
V_{\epsilon}^{x}\equiv\left(\mathcal{I}\left(h_{\sigma_{\epsilon}^{x}}\right),\mathcal{I}\left(h_{d\sigma_{\epsilon}^{x}}\right),\cdots,\mathcal{I}\left(h_{d^{n-1}\sigma_{\epsilon}^{x}}\right)\right)^{\top}.
\]
It turns out that in this setting we can also compute the covariance
matrix of the family $\left\{ V_{\epsilon}^{x}:x\in\mathbb{R}^{2n},\epsilon>0\right\} $
explicitly under each circumstance as prescribed in Lemma \ref{lem:on V_t(x)},
and the covariance matrix also has a similar ``separability'' property
as in four dimensional case. In fact, following a similar computation
as the one (provided in the appendix) conducted to prove Lemma \ref{lem:on V_t(x)},
it is not hard to see that there exist invertible $n\times n$ matrices
$\mathbf{A}\left(r\right)$, $\mathbf{B}\left(r\right)$, $\mathbf{C}\left(r\right)$
and $\mathbf{D}\left(r\right)$ for every $r\in\left(0,\infty\right)$,
such that all the entries of $\mathbf{A}\left(r\right)$ and $\mathbf{D}\left(r\right)$
are functions in the linear span of $\left\{ r^{-\ell}K_{k}\left(r\right):0\leq\ell\leq k\leq2n-2\right\} $,
while all the entries of $\mathbf{B}\left(r\right)$ and $\mathbf{C}\left(r\right)$
are in the linear span of $\left\{ r^{-\ell}I_{k}\left(r\right):0\leq\ell\leq k\leq2n-2\right\} $.
Moreover, given $x\in\mathbb{R}^{2n}$ and $\epsilon_{1}\geq\epsilon_{2}>0$,
\foreignlanguage{english}{$\mathbb{E}^{\mathcal{W}}\left[V_{\epsilon_{1}}^{x}\left(V_{\epsilon_{2}}^{x}\right)^{\top}\right]=\mathbf{A}\left(\epsilon_{1}\right)\mathbf{B}^{\top}\left(\epsilon_{2}\right)$;
given $x,y\in\mathbb{R}^{2n}$ with $x\neq y$ and $\epsilon_{1}>\left|x-y\right|+\epsilon_{2}$,
$\mathbb{E}^{\mathcal{W}}\left[V_{\epsilon_{1}}^{x}\left(V_{\epsilon_{2}}^{y}\right)^{\top}\right]=\mathbf{A}\left(\epsilon_{1}\right)\mathbb{\mathbf{C}}\left(\left|x-y\right|\right)\mathbf{B}^{\top}\left(\epsilon_{2}\right)$;
given $x,y\in\mathbb{R}^{2n}$ with $x\neq y$ and $\left|x-y\right|>\epsilon_{1}+\epsilon_{2}$,
$\mathbb{E}^{\mathcal{W}}\left[V_{\epsilon_{1}}^{x}\left(V_{\epsilon_{2}}^{y}\right)^{\top}\right]=\mathbf{B}\left(\epsilon_{1}\right)\mathbb{\mathbf{D}}\left(\left|x-y\right|\right)\mathbf{B}^{\top}\left(\epsilon_{2}\right)$.
Therefore, if we similarly defined the ``normalized'' vector $U_{\epsilon}^{x}\equiv\mathbf{B}^{-1}\left(\epsilon\right)V_{\epsilon}^{x}$,
then the Gaussian family $\left\{ U_{\epsilon}^{x}:x\in\mathbb{R}^{2n},\epsilon>0\right\} $
will have the same properties as those of the corresponding family
(also denoted by $U_{\epsilon}^{x}$) in four dimensions.}

\selectlanguage{english}%
On the other hand, all the entries of the matrix $\mathbf{B}\left(\epsilon\right)$
are linear combinations of $\epsilon^{-l}I_{k}\left(\epsilon\right)$
with $0\leq l\leq k\leq2n-2$, and if one lets $\epsilon_{2}\downarrow0$
in the covariance matrix obtained in the second circumstance (when
$\epsilon_{1}>\left|x-y\right|+\epsilon_{2}$) from above, combined
with integral expressions for the entries of the covariance matrix,
then one can easily conclude that there exists constant matrix $\mathbf{B}$
which is non-degenerate (hence so is $\mathbf{B}^{-1}$) such that
$\mathbf{B}\left(\epsilon\right)$ converges to $\mathbf{B}$ as $\epsilon\downarrow0$.
Therefore, all the entries of $\mathbf{B}^{-1}\left(\epsilon\right)$
must be analytic in $\epsilon$ near zero. In particular, by examining
the asymptotics of the entries of $\mathbf{B}^{-1}\left(\epsilon\right)$
near zero, one can find the appropriate constant vector $\zeta\in\mathbb{R}^{2n}$
such that $\left(U_{\epsilon}^{x},\zeta\right)_{\mathbb{R}^{2n}}$
``approximates'' the GFF at $x$ when $\epsilon$ is small in the
same sense as described in Section 2. 

\selectlanguage{american}%
Clearly, \foreignlanguage{english}{$\left(U_{\epsilon}^{x},\zeta\right)_{\mathbb{R}^{2n}}$}
has all the properties of $\mathcal{I}\left(h_{\mu_{\epsilon}^{x}}\right)$
in four dimensions as stated in Theorem \ref{thm:mu_epsilon}. \foreignlanguage{english}{When
}$\epsilon_{1}\geq\epsilon_{2}>0$,\foreignlanguage{english}{ 
\[
G\left(\epsilon_{1}\right)\equiv\mathbb{E}^{\mathcal{W}}\left[\left(U_{\epsilon_{1}}^{x},\zeta\right)_{\mathbb{R}^{2n}}^{2}\right]=\mathbb{E}^{\mathcal{W}}\left[\left(U_{\epsilon_{1}}^{x},\zeta\right)_{\mathbb{R}^{2n}}\left(U_{\epsilon_{2}}^{x},\zeta\right)_{\mathbb{R}^{2n}}\right];
\]
}when $\epsilon_{1}>\epsilon_{2}+\left|x-y\right|$, $\mathbb{E}^{\mathcal{W}}\left[\left(U_{\epsilon_{1}}^{x},\zeta\right)_{\mathbb{R}^{2n}}\left(U_{\epsilon_{2}}^{y},\zeta\right)_{\mathbb{R}^{2n}}\right]=c\left(\epsilon_{1},\left|x-y\right|\right)$
which is independent of $\epsilon_{2}$; when $\left|x-y\right|>\epsilon_{1}+\epsilon_{2}$,
$\mathbb{E}^{\mathcal{W}}\left[\left(U_{\epsilon_{1}}^{x},\zeta\right)_{\mathbb{R}^{2n}}\left(U_{\epsilon_{2}}^{y},\zeta\right)_{\mathbb{R}^{2n}}\right]=d\left(\left|x-y\right|\right)$
which is independent of $\epsilon_{1}$ and $\epsilon_{2}$. In principle,
we can derive the explicit formulas of $G\left(\epsilon_{1}\right)$,
$c\left(\epsilon_{1},\left|x-y\right|\right)$ and $d\left(\left|x-y\right|\right)$,
and one can expect that they have logarithmic growth when $\epsilon_{1}$
and $\left|x-y\right|$ are small because the Green's function of
the operator $\left(I-\Delta\right)^{n}$ on $\mathbb{R}^{2n}$ has
logarithmic growth near the diagonal. Therefore, it's reasonable to
believe that if one takes $\left(U_{\epsilon}^{x},\zeta\right)_{\mathbb{R}^{2n}}$
to construct a sequence of approximating measures, i.e., 
\[
m_{\epsilon_{k}}^{\theta}\left(dx\right)\equiv\exp\left(\gamma\left(U_{\epsilon_{k}}^{x},\zeta\right)_{\mathbb{R}^{2n}}\left(\theta\right)-\frac{\gamma^{2}}{2}G\left(\epsilon_{k}\right)\right)dx,
\]
then the sequence $\left\{ m_{\epsilon_{k}}^{\theta}:\, k\geq1\right\} $
will almost surely admit a limit measure in the sense of weak convergence.
Furthermore, the quantum scaling component of a bounded set on $\mathbb{R}^{2n}$
under this limit measure should also satisfy a quadratic relation
with its counterpart under the Lebesgue measure. However, the amount
and the complexity of computations quickly become considerable as
$n$ increases.

\subsection*{Generalizations to Manifolds:}

In this last part we explain a more conceptual approach to constructing
analogues of the two-dimensional GFF on compact even-dimensional manifolds.
As we have remarked in the introduction, in dimension two, the GFF
defines a measure on a conformal class of metrics on a Riemann surface
$\Sigma$, constructed starting with a reference metric $g_{0}$ on
$\Sigma$, but in the end independent of $g_{0}$. In fact, the GFF
inner product of two functions $f_{1},f_{2}\in C_{c}^{\infty}(\Sigma)$
is defined by (\cite{Shef,DS1,DS2,HMP}) 
\[
(f_{1},f_{2})_{\Delta_{g_{0}}}\equiv(f_{1},\Delta_{g_{0}}f_{2})\equiv\int_{\Sigma}f_{1}(x)(\Delta_{g_{0}}f_{2})(x)d\textrm{vol}_{g_{0}}(x),
\]
where $\Delta_{g_{0}}$ is the Laplace-Beltrami operator on $\Sigma$
with respect to $g_{0}$. This inner product is conformally invariant.
Indeed, if the metric $g_{0}$ is changed conformally to $g_{1}=e^{2\omega}g_{0}$
for some $\omega\in C_{c}^{\infty}\left(\Sigma\right)$, then the
volume element changes as 
\[
d\textrm{vol}_{g_{1}}=e^{2\omega}d\textrm{vol}_{g_{0}},
\]
while the Laplacian is changed as $\Delta_{g_{1}}=e^{-2\omega}\Delta_{g_{0}}$.
Therefore, after obvious cancellations we find that 
\[
(f_{1},f_{2})_{\Delta_{g_{1}}}=(f_{1},f_{2})_{\Delta_{g_{0}}}.
\]

It seems natural to define a similar measure for conformal classes
of metrics in higher dimensions. Below, we explain how to do that
for certain conformal classes on compact manifolds $\mathfrak{M}$
of even dimension $2n$. In the construction, we find it convenient
to use \emph{conformally covariant} elliptic operators described below.
In the discussion, we restrict ourselves to even-dimensional manifolds,
although the corresponding operators can be defined in odd dimensions
as well.

Let $\mathfrak{M}$ be a manifold of even dimension $2n$, $n\geq2$,
and $g_{0}$ a Riemannian metric on $\mathfrak{M}$. Then, there exists
on $\mathfrak{M}$ an elliptic operator $P=P_{g_{0}}$ of order $2n$,
called the \emph{dimension-critical GJMS operator}, constructed by
Graham-Jenne-Mason-Sparling in \cite{GJMS}, with the following properties:
\[
P=\Delta^{n}+\mbox{ lower order terms};
\]
in fact, $P$ has a polynomial expression in (Levi-Civita connection)
$\nabla$ and (scalar curvature) $R$, with coefficients that are
rational in dimension $2n$; $P$ is formally self-adjoint (\cite{GZ,FG});
under a conformal change of metric $g_{1}=e^{2\omega}g_{0}$, the
operator $P$ changes as $P_{g_{1}}=e^{-2n\omega}P_{g_{0}}$. 

Given these properties, we can imitate the construction of the GFF
in dimension $2$: for $f_{1},f_{2}\in C^{\infty}(\mathfrak{M})$,
the inner product is defined by 
\[
(f_{1},f_{2})_{P_{g_{0}}}\equiv\int_{\mathfrak{M}}f_{1}(x)(P_{g_{0}}f_{2})(x)d\textrm{vol}_{g_{0}}(x).
\]
Then, this inner product is also conformally invariant. When the metric
$g_{0}$ is changed conformally to $g_{1}=e^{2\omega}g_{0}$, the
volume element changes as 
\[
d\textrm{vol}_{g_{1}}=e^{2n\omega}d\textrm{vol}_{g_{0}},
\]
while $P$ changes as $P_{g_{1}}=e^{-2n\omega}P_{g_{0}}$. Again,
we get the relation 
\[
(f_{1},f_{2})_{P_{g_{1}}}=(f_{1},f_{2})_{P_{g_{0}}},
\]
 just like in dimension two.

When $n=2$, the dimension-critical GJMS operator 
\[
P_{4}=\Delta_{g_{0}}^{2}+\delta[(2/3)R_{g_{0}}g_{0}-2{\rm Ric}_{g_{0}}]d
\]
is also called the \emph{Paneitz operator}. If $\mathfrak{M}$ is
flat, then the Paneitz operator is equal to $\Delta^{2}$, hence in
$\mathbb{R}^{4}$ it is natural to work with $\Delta^{2}$. However,
since $\mathbb{R}^{4}$ is not compact, we need to consider the operator
on a compact domain, in which case we have to choose proper boundary
operators in order to preserve the conformal covariance property.
This will be further explored in future work.

On the compact $2n$-dimensional manifold $\mathfrak{M}$, if we construct
a Gaussian random field using the dimension-critical GJMS operator
$P$, then the covariance function of the field is given by the Green's
function $G_{P}(x,y)$ of the operator $P$. Let $d\left(x,y\right)$
be the Riemannian distance between $x$ and $y$ on $\mathfrak{M}$.
Then, it is known (\cite{CY,Nd,Ponge}) that as $d\left(x,y\right)\downarrow0$,
$G_{P}(x,y)$ is asymptotic to $-C_{n}\log d(x,y)$ where $C_{n}>0$
depends only on the dimension. This is similar to the well-known behavior
of the Green's function of the Laplace-Beltrami operator $\Delta$
in dimension two. This will become an important ingredient in the
construction of the random measure on the manifold, which we intend
to explore in a future paper.

\section{Appendix}

This section contains all the computations with the Bessel functions.
We start with the Fourier transforms of $\sigma_{\epsilon}^{x}$ and
$d\sigma_{\epsilon}^{x}$, and list all the integral expressions for
the covariance function of the family $\left\{ \mathcal{I}\left(h_{\sigma_{\epsilon}^{x}}\right),\mathcal{I}\left(h_{d\sigma_{\epsilon}^{x}}\right):\, x\in\mathbb{R}^{4},\epsilon>0\right\} $.
\begin{lem}
\label{lem:integral forms for cov}Recall from (\ref{eq:fourier transf. of average})
and (\ref{eq:fourier transf. for derivative of average}) that the
Fourier transforms of $\sigma_{\epsilon}^{x}$ and $d\sigma_{\epsilon}^{x}$
are, respectively,
\[
\hat{\sigma_{\epsilon}^{x}}\left(\xi\right)=2\left(\epsilon\left|\xi\right|\right)^{-1}J_{1}\left(\epsilon\left|\xi\right|\right)e^{i\left(x,\xi\right)_{\mathbb{R}^{4}}}
\]
\foreignlanguage{english}{\textup{
\[
\mbox{and }\hat{d\sigma_{\epsilon}^{x}}\left(\xi\right)=\frac{d}{d\epsilon}\hat{\sigma_{\epsilon}^{x}}\left(\xi\right)=-2\epsilon^{-1}J_{2}\left(\epsilon\left|\xi\right|\right)e^{i\left(x,\xi\right)_{\mathbb{R}^{4}}}.
\]
}}Therefore, both $\sigma_{\epsilon}^{x}$ and $d\sigma_{\epsilon}^{x}$
are in $H^{-2}\left(\mathbb{R}^{4}\right)$. In fact, for $\epsilon_{1},\epsilon_{2}>0$,
\begin{equation}
\mathbb{E}^{\mathcal{W}}\left[\mathcal{I}\left(h_{\sigma_{\epsilon_{1}}^{x}}\right)\mathcal{I}\left(h_{\sigma_{\epsilon_{2}}^{x}}\right)\right]=\frac{1}{2\pi^{2}\epsilon_{1}\epsilon_{2}}\int_{0}^{\infty}\frac{\tau}{\left(1+\tau^{2}\right)^{2}}J_{1}\left(\epsilon_{1}\tau\right)J_{1}\left(\epsilon_{2}\tau\right)d\tau,\label{eq:var sphe ave integral form}
\end{equation}
 
\begin{equation}
\mathbb{E}^{\mathcal{W}}\left[\mathcal{I}\left(h_{\sigma_{\epsilon_{1}}^{x}}\right)\mathcal{I}\left(h_{d\sigma_{\epsilon_{2}}^{x}}\right)\right]=\frac{-1}{2\pi^{2}\epsilon_{1}\epsilon_{2}}\int_{0}^{\infty}\frac{\tau^{2}}{\left(1+\tau^{2}\right)^{2}}J_{1}\left(\epsilon_{1}\tau\right)J_{2}\left(\epsilon_{2}\tau\right)d\tau,\label{eq:var sphe ave mix integral form}
\end{equation}
and 
\begin{equation}
\mathbb{E}^{\mathcal{W}}\left[\mathcal{I}\left(h_{d\sigma_{\epsilon_{1}}^{x}}\right)\mathcal{I}\left(h_{d\sigma_{\epsilon_{2}}^{x}}\right)\right]=\frac{1}{2\pi^{2}\epsilon_{1}\epsilon_{2}}\int_{0}^{\infty}\frac{\tau^{3}}{\left(1+\tau^{2}\right)^{2}}J_{2}\left(\epsilon_{1}\tau\right)J_{2}\left(\epsilon_{2}\tau\right)d\tau.\label{eq:var sphe ave deriv integral form}
\end{equation}
Furthermore, for $x,y\in\mathbb{R}^{4}$, $x\neq y$, and $\epsilon_{1},\epsilon_{2}>0$,
\begin{multline}
\begin{split}\mathbb{E}^{\mathcal{W}}\left[\mathcal{I}\left(h_{\sigma_{\epsilon_{1}}^{x}}\right)\mathcal{I}\left(h_{\sigma_{\epsilon_{2}}^{y}}\right)\right]\qquad\qquad\qquad\qquad\qquad\qquad\qquad\qquad\qquad\qquad\\
=\frac{1}{\pi^{2}\epsilon_{1}\epsilon_{2}\left|x-y\right|}\int_{0}^{\infty}\frac{1}{\left(1+\tau^{2}\right)^{2}}J_{1}\left(\epsilon_{1}\tau\right)J_{1}\left(\epsilon_{2}\tau\right)J_{1}\left(\left|x-y\right|\tau\right)d\tau,
\end{split}
\label{eq:cov sphe ave integral form}
\end{multline}
 
\begin{equation}
\begin{split}\mathbb{E}^{\mathcal{W}}\left[\mathcal{I}\left(h_{\sigma_{\epsilon_{1}}^{x}}\right)\mathcal{I}\left(h_{d\sigma_{\epsilon_{2}}^{y}}\right)\right]\qquad\qquad\qquad\qquad\qquad\qquad\qquad\qquad\qquad\qquad\\
=\frac{-1}{\pi^{2}\epsilon_{1}\epsilon_{2}\left|x-y\right|}\int_{0}^{\infty}\frac{\tau}{\left(1+\tau^{2}\right)^{2}}J_{1}\left(\epsilon_{1}\tau\right)J_{2}\left(\epsilon_{2}\tau\right)J_{1}\left(\left|x-y\right|\tau\right)d\tau,
\end{split}
\label{eq:cov sphe ave mix integral form}
\end{equation}
and 
\begin{equation}
\begin{split}\mathbb{E}^{\mathcal{W}}\left[\mathcal{I}\left(h_{d\sigma_{\epsilon_{1}}^{x}}\right)\mathcal{I}\left(h_{d\sigma_{\epsilon_{2}}^{y}}\right)\right]\qquad\qquad\qquad\qquad\qquad\qquad\qquad\qquad\qquad\qquad\\
=\frac{1}{\pi^{2}\epsilon_{1}\epsilon_{2}\left|x-y\right|}\int_{0}^{\infty}\frac{\tau^{2}}{\left(1+\tau^{2}\right)^{2}}J_{2}\left(\epsilon_{1}\tau\right)J_{2}\left(\epsilon_{2}\tau\right)J_{1}\left(\left|x-y\right|\tau\right)d\tau.
\end{split}
\label{eq:cov sphe ave deriv integral form}
\end{equation}
\end{lem}
\begin{proof}
Everything follows from straightforward computations in spherical
coordinates in $\mathbb{R}^{4}$ and applications of the following
integral expression of the Bessel functions (\cite{BesselFunctions},
§3.3 ): for every $k\geq1$ and $r>0$, 
\[
J_{k}\left(r\right)=\frac{\left(\frac{r}{2}\right)^{k}}{\Gamma\left(k+\frac{1}{2}\right)\Gamma\left(\frac{1}{2}\right)}\int_{0}^{\pi}e^{ir\cos\left(\theta\right)}\sin^{2k}\theta d\theta.
\]
In addition, as we have indicated in Section 2, the asymptotic expansion
(\cite{BesselFunctions}, §7.1) of $J_{k}$ says that $J_{k}\left(r\right)=\mathcal{O}\left(r^{-1/2}\right)$
as $r\rightarrow\infty$, which is sufficient to guarantee the convergence
of each integral involved in (\ref{eq:var sphe ave integral form})-(\ref{eq:cov sphe ave deriv integral form}). 
\end{proof}
It will be convenient to recognize that all the covariance functions
involved in the previous lemma, i.e., 
\[
\mathbb{E}^{\mathcal{W}}\left[\mathcal{I}\left(h_{\sigma_{\epsilon_{1}}^{x}}\right)\mathcal{I}\left(h_{\sigma_{\epsilon_{2}}^{y}}\right)\right],\,\mathbb{E}^{\mathcal{W}}\left[\mathcal{I}\left(h_{\sigma_{\epsilon_{1}}^{x}}\right)\mathcal{I}\left(h_{d\sigma_{\epsilon_{2}}^{y}}\right)\right]\mbox{ and }\mathbb{E}^{\mathcal{W}}\left[\mathcal{I}\left(h_{d\sigma_{\epsilon_{1}}^{x}}\right)\mathcal{I}\left(h_{d\sigma_{\epsilon_{2}}^{y}}\right)\right],
\]
are continuous in all variables $x,y\in\mathbb{R}^{4}$ and $\epsilon_{1},\epsilon_{2}>0$.
In fact, 
\[
\mathbb{E}^{\mathcal{W}}\left[\mathcal{I}\left(h_{\sigma_{\epsilon_{1}}^{x}}\right)\mathcal{I}\left(h_{d\sigma_{\epsilon_{2}}^{y}}\right)\right]=\frac{d}{d\epsilon_{2}}\mathbb{E}\left[\mathcal{I}\left(h_{\sigma_{\epsilon_{1}}^{x}}\right)\mathcal{I}\left(h_{\sigma_{\epsilon_{2}}^{y}}\right)\right]
\]
and 
\[
\mathbb{E}^{\mathcal{W}}\left[\mathcal{I}\left(h_{d\sigma_{\epsilon_{1}}^{x}}\right)\mathcal{I}\left(h_{d\sigma_{\epsilon_{2}}^{y}}\right)\right]=\frac{d^{2}}{d\epsilon_{1}d\epsilon_{2}}\mathbb{E}\left[\mathcal{I}\left(h_{\sigma_{\epsilon_{1}}^{x}}\right)\mathcal{I}\left(h_{\sigma_{\epsilon_{2}}^{y}}\right)\right].
\]
These simply follow from the dominated convergence theorem and the
fact that both $J_{k}\left(r\right)$ and $J_{k}\left(r\right)/r$
are bounded on $r\in\left(0,\infty\right)$ for every $k\geq1$ . 

$\quad$\\
\emph{Proof of Lemma \ref{lem:on V_t(x)}:} The proof of all the formulas
(\ref{eq:vector covariance concentric}) to (\ref{eq:vector covariance nonoverlap})
is based on the following integral formulas of Bessel functions which
can be found in \cite{BesselFunctions}, §13.53, pp 429-430: if $a\geq b>0$
and $p>0$, then 
\begin{equation}
\int_{0}^{\infty}\frac{\tau}{\tau^{2}+p^{2}}J_{1}\left(a\tau\right)J_{1}\left(b\tau\right)d\tau=K_{1}\left(ap\right)I_{1}\left(bp\right);\label{eq:bessel fcn integral-1}
\end{equation}
if $a>b+c$ and $p>0$, then
\begin{equation}
\int_{0}^{\infty}\frac{1}{\tau^{2}+p^{2}}J_{1}\left(a\tau\right)J_{1}\left(b\tau\right)J_{1}\left(c\tau\right)d\tau=p^{-1}K_{1}\left(ap\right)I_{1}\left(bp\right)I_{1}\left(cp\right).\label{eq:bessel fcn integral-2}
\end{equation}
We hereby provide an alternative proof of these two formulas, and
complete the computations in Lemma \ref{lem:on V_t(x)}. For a fixed
$p>0$, we define the function 
\[
B\left(a,b\right)\equiv\frac{1}{2\pi^{2}ab}\int_{0}^{\infty}\frac{\tau}{\tau^{2}+p^{2}}J_{1}\left(a\tau\right)J_{1}\left(b\tau\right)d\tau\mbox{ for }a,b>0.
\]
Given $a>0$, we observe that if $\delta_{x}$ is the point mass at
$x\in\mathbb{R}^{4}$, then the following integral is finite:
\begin{align*}
 & \left(\frac{1}{2\pi}\right)^{4}\int_{\mathbb{R}^{4}}\frac{1}{p^{2}+\left|\xi\right|^{2}}\cdot e^{i\left(x,\xi\right)}\cdot\frac{2J_{1}\left(a\left|\xi\right|\right)}{a\left|\xi\right|}d\xi\\
= & \frac{1}{2\pi^{2}a\left|x\right|}\int_{0}^{\infty}\frac{\tau J_{1}\left(a\tau\right)J_{1}\left(\left|x\right|\tau\right)}{p^{2}+\tau^{2}}d\tau=B\left(a,\left|x\right|\right).
\end{align*}
But this integral is also ``formally'' equal to $\left(\left(p^{2}-\Delta\right)^{-1}\sigma_{a}^{0},\delta_{x}\right)_{L^{2}}$.
In other words, 
\[
\left(\left(p^{2}-\Delta\right)^{-1}\sigma_{a}^{0}\right)\left(x\right)=B\left(a,\left|x\right|\right)
\]
is a point-wise defined, radially symmetric function in $x\in\mathbb{R}^{4}$.
Therefore, in the sense of tempered distribution,
\[
\left(p^{2}-\Delta\right)B\left(a,\left|x\right|\right)=\sigma_{a}^{0},
\]
which, when written in spherical coordinates, implies
\[
\left(p^{2}-\partial_{b}^{2}-\frac{3}{b}\partial_{b}\right)B\left(a,b\right)=0\mbox{ for all }0<b\neq a.
\]
\foreignlanguage{english}{The above is a Bessel-type ordinary differential
equation, all the solutions of which are in the form of
\[
C_{1}\left(a\right)\frac{K_{1}\left(bp\right)}{b}+C_{2}\left(a\right)\frac{I_{1}\left(bp\right)}{b},
\]
where $C_{1}$ and $C_{2}$ are two functions only depending on $a$.
Without loss of generality, we can assume $b<a$. If one examines
the behavior of $B\left(a,b\right)$ when $b$ is close to zero, then
one finds that $C_{1}\left(a\right)\equiv0$ because $bB\left(a,b\right)$
converges to zero while $K_{1}\left(bp\right)$ blows up as $b\downarrow0$.
On the other hand, with $b>0$ fixed, one can apply exactly the same
arguments to see that $B\left(a,b\right)$ also satisfies 
\[
\left(p^{2}-\partial_{a}^{2}-\frac{3}{a}\partial_{a}\right)B\left(a,b\right)=0\mbox{ for all }a>b.
\]
Hence, $C_{2}\left(a\right)$ must be in the form of }

\selectlanguage{english}%
\[
C_{2}\left(a\right)=C\frac{K_{1}\left(ap\right)}{a}+C^{\prime}\frac{I_{1}\left(ap\right)}{a}
\]
for some constant $C$ and $C^{\prime}$. This time, the boundedness
of $aB\left(a,b\right)$ as $a\uparrow\infty$ implies $C^{\prime}=0$. 

Thus, the only thing left is to determine the constant $C$. To this
end, we observe
\[
\begin{split}2\pi^{2}a^{2}B\left(a,a\right) & =\int_{0}^{\infty}\frac{u}{u^{2}+a^{2}p^{2}}J_{1}^{2}\left(u\right)du\\
 & \longrightarrow\int_{0}^{\infty}\frac{J_{1}^{2}\left(u\right)}{u}du\mbox{ as }a\downarrow0.
\end{split}
\]
However, one can easily verify that 
\[
\frac{d}{du}\left(-\frac{J_{0}^{2}\left(u\right)+J_{1}^{2}\left(u\right)}{2}\right)=\frac{J_{1}^{2}\left(u\right)}{u}.
\]
So $\lim_{a\downarrow0}2\pi^{2}a^{2}B\left(a,a\right)=\frac{1}{2}$.
Meanwhile, $\lim_{a\downarrow0}I_{1}\left(ap\right)K_{1}\left(ap\right)=\frac{1}{2}$,
which implies $C=\frac{1}{2\pi^{2}}$. Therefore, 
\[
B\left(a,b\right)=\frac{1}{2\pi^{2}ab}I_{1}\left(bp\right)K_{1}\left(ap\right)\mbox{ for }a>b>0.
\]

For the formula (\ref{eq:bessel fcn integral-2}), we define 
\[
C\left(a,b,c\right)=\frac{1}{\pi^{2}abc}\int_{0}^{\infty}\frac{1}{\tau^{2}+p^{2}}J_{1}\left(a\tau\right)J_{1}\left(b\tau\right)J_{1}\left(c\tau\right)d\tau\mbox{ for }a,b,c>0.
\]
Assume $a>b+c$. One can verify, by direct computations inside the
integral signs and the dominated convergence theorem, that
\[
p^{2}C\left(a,b,c\right)-\frac{\partial^{2}}{\partial c^{2}}C\left(a,b,c\right)-\frac{3}{c}\frac{\partial}{\partial c}C\left(a,b,c\right)=0\mbox{ for }0<c<a-b,
\]
\[
\mbox{ and }\lim_{c\downarrow0}C\left(a,b,c\right)=B\left(a,b\right).
\]
Similarly as above, one has 
\[
C\left(a,b,c\right)=\frac{2I_{1}\left(cp\right)}{cp}B\left(a,b\right)=\frac{1}{\pi^{2}abcp}K_{1}\left(ap\right)I_{1}\left(bp\right)I_{1}\left(cp\right).
\]
Thus, (\ref{eq:bessel fcn integral-1}) and (\ref{eq:bessel fcn integral-2})
are proved. 

\selectlanguage{american}%
Given (\ref{eq:bessel fcn integral-1}), notice that $\mathbb{E}^{\mathcal{W}}\left[\mathcal{I}\left(h_{\sigma_{\epsilon_{1}}^{x}}\right)\mathcal{I}\left(h_{\sigma_{\epsilon_{2}}^{x}}\right)\right]$
can be computed by applying the operator $\frac{-1}{2p}\frac{d}{dp}|_{p=1}$
to both sides of (\ref{eq:bessel fcn integral-1}) with $a=\epsilon_{1}\vee\epsilon_{2}$
and $b=\epsilon_{1}\wedge\epsilon_{2}$. Then from there, based on
the earlier observations, the complete expression for the covariance
matrix in the concentric case can be obtained by taking derivatives
in $\epsilon_{1}$ and $\epsilon_{2}$ accordingly. The detailed computations
are as follows.\foreignlanguage{english}{ When $\epsilon_{1}\geq\epsilon_{2}>0$,
\[
\begin{split}\mathbb{E}^{\mathcal{W}}\left[\mathcal{I}\left(h_{\sigma_{\epsilon_{1}}^{x}}\right)\mathcal{I}\left(h_{\sigma_{\epsilon_{2}}^{x}}\right)\right] & =\left(-\frac{1}{2p}\frac{d}{dp}|_{p=1}\right)B\left(\epsilon_{1},\epsilon_{2}\right)\\
 & =\frac{-1}{4\pi^{2}}\left(K_{1}^{\prime}\left(\epsilon_{1}\right)\frac{I_{1}\left(\epsilon_{2}\right)}{\epsilon_{2}}+\frac{K_{1}\left(\epsilon_{1}\right)}{\epsilon_{1}}I_{1}^{\prime}\left(\epsilon_{2}\right)\right)\\
 & =\frac{-1}{4\pi^{2}}\left(\begin{array}{cc}
K_{1}^{\prime}\left(\epsilon_{1}\right) & K_{1}\left(\epsilon_{1}\right)/\epsilon_{1}\end{array}\right)\left(\begin{array}{c}
I_{1}\left(\epsilon_{2}\right)/\epsilon_{2}\\
I_{1}^{\prime}\left(\epsilon_{2}\right)
\end{array}\right).
\end{split}
\]
Thus,
\[
\begin{split} & \mathbb{E}^{\mathcal{W}}\left[V_{\epsilon_{1}}^{x}\left(V_{\epsilon_{2}}^{x}\right)^{\top}\right]\\
= & \left(\begin{array}{cc}
1 & \frac{\partial}{\partial\epsilon_{2}}\\
\frac{\partial}{\partial\epsilon_{1}} & \frac{\partial^{2}}{\partial\epsilon_{1}\partial\epsilon_{2}}
\end{array}\right)\mathbb{E}^{\mathcal{W}}\left[\mathcal{I}\left(h_{\sigma_{\epsilon_{1}}^{x}}\right)\mathcal{I}\left(h_{\sigma_{\epsilon_{2}}^{x}}\right)\right]\\
= & \frac{-1}{4\pi^{2}}\left(\begin{array}{cc}
K_{1}^{\prime}\left(\epsilon_{1}\right) & K_{1}\left(\epsilon_{1}\right)/\epsilon_{1}\\
K^{\prime\prime}\left(\epsilon_{1}\right) & \left(K_{1}\left(\epsilon_{1}\right)/\epsilon_{1}\right)^{\prime}
\end{array}\right)\left(\begin{array}{cc}
I_{1}\left(\epsilon_{2}\right)/\epsilon_{2} & \left(I_{1}\left(\epsilon_{2}\right)/\epsilon_{2}\right)^{\prime}\\
I_{1}^{\prime}\left(\epsilon_{2}\right) & I_{1}^{\prime\prime}\left(\epsilon_{2}\right)
\end{array}\right).
\end{split}
\]
Besides, one realizes that 
\[
\left(\frac{K_{1}\left(\epsilon_{1}\right)}{\epsilon_{1}}\right)^{\prime}=-\frac{K_{2}\left(\epsilon_{1}\right)}{\epsilon_{1}}\mbox{ and }\left(\frac{I_{1}\left(\epsilon_{2}\right)}{\epsilon_{2}}\right)^{\prime}=\frac{I_{2}\left(\epsilon_{2}\right)}{\epsilon_{2}},
\]
and the formula (\ref{eq:vector covariance concentric}) follows.}

The non-concentric case is very similar. Given (\ref{eq:bessel fcn integral-2}),
we assign $a=\epsilon_{1}$ in case (2) and $a=\left|x-y\right|$
in case (3). By a similar procedure, i.e., applying $\frac{-1}{2p}\frac{d}{dp}|_{p=1}$
to (\ref{eq:bessel fcn integral-2}) and taking derivatives in $\epsilon_{1}$
and $\epsilon_{2}$, we will be able to compute the non-concentric
covariance matrix in either (2) or (3). To be specific, when $\epsilon_{1}>\left|x-y\right|+\epsilon_{2}$,
$\mathbb{E}^{\mathcal{W}}\left[\mathcal{I}\left(h_{\sigma_{\epsilon_{1}}^{x}}\right)\mathcal{I}\left(h_{\sigma_{\epsilon_{2}}^{y}}\right)\right]$
equals 
\[
\begin{split} & \left(-\frac{1}{2p}\frac{d}{dp}|_{p=1}\right)C\left(\epsilon_{1},\epsilon_{2},\left|x-y\right|\right)\\
= & \frac{-1}{2\pi^{2}}\left(K_{1}^{\prime}\left(\epsilon_{1}\right)\frac{I_{1}\left(\epsilon_{2}\right)}{\epsilon_{2}}\frac{I_{1}\left(\left|x-y\right|\right)}{\left|x-y\right|}+\frac{K_{1}\left(\epsilon_{1}\right)}{\epsilon_{1}}I_{1}^{\prime}\left(\epsilon_{2}\right)\frac{I_{1}\left(\left|x-y\right|\right)}{\left|x-y\right|}\right)\\
 & \qquad\qquad\qquad\qquad-\frac{1}{2\pi^{2}}\left(\frac{K_{1}\left(\epsilon_{1}\right)}{\epsilon_{1}}\frac{I_{1}\left(\epsilon_{2}\right)}{\epsilon_{2}}I_{1}^{\prime}\left(\left|x-y\right|\right)-\frac{K_{1}\left(\epsilon_{1}\right)}{\epsilon_{1}}\frac{I_{1}\left(\epsilon_{2}\right)}{\epsilon_{2}}\frac{I_{1}\left(\left|x-y\right|\right)}{\left|x-y\right|}\right)\\
= & \frac{-1}{2\pi^{2}}\left(K_{1}^{\prime}\left(\epsilon_{1}\right)\frac{I_{1}\left(\epsilon_{2}\right)}{\epsilon_{2}}\frac{I_{1}\left(\left|x-y\right|\right)}{\left|x-y\right|}+\frac{K_{1}\left(\epsilon_{1}\right)}{\epsilon_{1}}I_{1}^{\prime}\left(\epsilon_{2}\right)\frac{I_{1}\left(\left|x-y\right|\right)}{\left|x-y\right|}+\frac{K_{1}\left(\epsilon_{1}\right)}{\epsilon_{1}}\frac{I_{1}\left(\epsilon_{2}\right)}{\epsilon_{2}}I_{2}\left(\left|x-y\right|\right)\right)\\
= & \frac{-1}{2\pi^{2}}\left(\begin{array}{cc}
K_{1}^{\prime}\left(\epsilon_{1}\right) & \frac{K_{1}\left(\epsilon_{1}\right)}{\epsilon_{1}}\end{array}\right)\left(\begin{array}{cc}
\frac{I_{1}\left(\left|x-y\right|\right)}{\left|x-y\right|} & 0\\
I_{2}\left(\left|x-y\right|\right) & \frac{I_{1}\left(\left|x-y\right|\right)}{\left|x-y\right|}
\end{array}\right)\left(\begin{array}{c}
I_{1}\left(\epsilon_{2}\right)/\epsilon_{2}\\
I_{1}^{\prime}\left(\epsilon_{2}\right)
\end{array}\right);
\end{split}
\]
when $\left|x-y\right|>\epsilon_{1}+\epsilon_{2}$, $\mathbb{E}^{\mathcal{W}}\left[\mathcal{I}\left(h_{\sigma_{\epsilon_{1}}^{x}}\right)\mathcal{I}\left(h_{\sigma_{\epsilon_{2}}^{y}}\right)\right]$
equals 
\[
\begin{split} & \left(-\frac{1}{2p}\frac{d}{dp}|_{p=1}\right)C\left(\left|x-y\right|,\epsilon_{1},\epsilon_{2}\right)\\
= & \frac{-1}{2\pi^{2}}\left(K_{1}^{\prime}\left(\left|x-y\right|\right)\frac{I_{1}\left(\epsilon_{1}\right)}{\epsilon_{1}}\frac{I_{1}\left(\epsilon_{2}\right)}{\epsilon_{2}}-\frac{K_{1}\left(\left|x-y\right|\right)}{\left|x-y\right|}\frac{I_{1}\left(\epsilon_{1}\right)}{\epsilon_{1}}\frac{I_{1}\left(\epsilon_{2}\right)}{\epsilon_{2}}\right)\\
 & \qquad\qquad\qquad\qquad-\frac{1}{2\pi^{2}}\left(\frac{K_{1}\left(\left|x-y\right|\right)}{\left|x-y\right|}\frac{I_{1}\left(\epsilon_{1}\right)}{\epsilon_{1}}I_{1}^{\prime}\left(\epsilon_{2}\right)+\frac{K_{1}\left(\left|x-y\right|\right)}{\left|x-y\right|}I_{1}^{\prime}\left(\epsilon_{1}\right)\frac{I_{1}\left(\epsilon_{2}\right)}{\epsilon_{2}}\right)\\
= & \frac{-1}{2\pi^{2}}\left(-K_{2}\left(\left|x-y\right|\right)\frac{I_{1}\left(\epsilon_{1}\right)}{\epsilon_{1}}\frac{I_{1}\left(\epsilon_{2}\right)}{\epsilon_{2}}+\frac{K_{1}\left(\left|x-y\right|\right)}{\left|x-y\right|}\frac{I_{1}\left(\epsilon_{1}\right)}{\epsilon_{1}}I_{1}^{\prime}\left(\epsilon_{2}\right)+\frac{K_{1}\left(\left|x-y\right|\right)}{\left|x-y\right|}I_{1}^{\prime}\left(\epsilon_{1}\right)\frac{I_{1}\left(\epsilon_{2}\right)}{\epsilon_{2}}\right)\\
= & \frac{-1}{2\pi^{2}}\left(\begin{array}{cc}
\frac{I_{1}\left(\epsilon_{1}\right)}{\epsilon_{1}} & I_{1}^{\prime}\left(\epsilon_{1}\right)\end{array}\right)\left(\begin{array}{cc}
-K_{2}\left(\left|x-y\right|\right) & \frac{K_{1}\left(\left|x-y\right|\right)}{\left|x-y\right|}\\
\frac{K_{1}\left(\left|x-y\right|\right)}{\left|x-y\right|} & 0
\end{array}\right)\left(\begin{array}{c}
I_{1}\left(\epsilon_{2}\right)/\epsilon_{2}\\
I_{1}^{\prime}\left(\epsilon_{2}\right)
\end{array}\right).
\end{split}
\]
The rest is straightforward. Thus we have finished the proof of Lemma
\ref{lem:on V_t(x)}. $\square$

Before continuing, we point out that similar computations can be carried
out in higher even dimensions $\mathbb{R}^{2n}$ with $n\geq2$. In
fact, the formulas (\ref{eq:bessel fcn integral-1}) and (\ref{eq:bessel fcn integral-2})
remain true if one replaces $J_{1}$ by $J_{n-1}$ in the integrands,
times a factor of $\tau^{2-n}$ to the integrand of (\ref{eq:bessel fcn integral-2}),
and replaces $K_{1}$ by $K_{n-1}$, $I_{1}$ by $I_{n-1}$ respectively
in the results. Therefore, one can use these modified results to compute
the covariance function of the Gaussian family $\left\{ \mathcal{I}\left(h_{\sigma_{\epsilon}^{x}}\right):x\in\mathbb{R}^{2n},\epsilon>0\right\} $
(as defined in Section 6.1) by applying the operator $\left(-\frac{1}{2p}\frac{d}{dp}\right)^{n-1}|_{p=1}$
to the modified version of (\ref{eq:bessel fcn integral-1}) and (\ref{eq:bessel fcn integral-2}).
The rest follows similarly as above. 

Next, we want to provide details in the deriving the formulas for
$\mu_{x}^{\epsilon}$ (\ref{eq:mu_epsilon def}) and $G\left(\epsilon\right)$
(\ref{eq: variance G}) as well as the results (\ref{eq:cov concentric})-(\ref{eq:cov nonoverlap})
in Theorem \ref{eq:mu_epsilon def}. It's an easy matter to check
that for every $\epsilon>0$, 
\[
\det\mathbf{B}\left(\epsilon\right)=\epsilon^{-1}\left(I_{1}^{2}\left(\epsilon\right)-I_{0}\left(\epsilon\right)I_{2}\left(\epsilon\right)\right)>0,
\]
where we applied the Bessel function identities (\cite{BesselFunctions},
§3.71) $I_{1}^{\prime}\left(\epsilon\right)=\frac{-I_{1}\left(\epsilon\right)}{\epsilon}+I_{0}\left(\epsilon\right)$
and $I_{1}^{\prime\prime}\left(\epsilon\right)=\frac{-I_{2}\left(\epsilon\right)}{\epsilon}+I_{1}\left(\epsilon\right).$
Therefore, 
\[
\mathbf{B}^{-1}\left(\epsilon\right)=\frac{1}{I_{1}^{2}\left(\epsilon\right)-I_{0}\left(\epsilon\right)I_{2}\left(\epsilon\right)}\left(\begin{array}{cc}
\epsilon I_{1}\left(\epsilon\right)-I_{2}\left(\epsilon\right) & I_{1}\left(\epsilon\right)-\epsilon I_{0}\left(\epsilon\right)\\
-I_{2}\left(\epsilon\right) & I_{1}\left(\epsilon\right)
\end{array}\right).
\]
Recall that $U_{\epsilon}^{x}=\mathbf{B}^{-1}\left(\epsilon\right)V_{\epsilon}^{x}$,
when computed explicitly, 
\[
U_{\epsilon}^{x}=\frac{1}{I_{1}^{2}\left(\epsilon\right)-I_{0}\left(\epsilon\right)I_{2}\left(\epsilon\right)}\left(\begin{array}{c}
\left(\epsilon I_{1}\left(\epsilon\right)-I_{2}\left(\epsilon\right)\right)\mathcal{I}\left(h_{\sigma_{\epsilon}^{x}}\right)+\left(I_{1}\left(\epsilon\right)-\epsilon I_{0}\left(\epsilon\right)\right)\mathcal{I}\left(h_{d\sigma_{\epsilon}^{x}}\right)\\
-I_{2}\left(\epsilon\right)\mathcal{I}\left(h_{\sigma_{\epsilon}^{x}}\right)+I_{1}\left(\epsilon\right)\mathcal{I}\left(h_{d\sigma_{\epsilon}^{x}}\right)
\end{array}\right),
\]
and if $\zeta=\left(1,1\right)^{\top}$, then $\mu_{\epsilon}^{x}=f_{1}\left(\epsilon\right)\sigma_{\epsilon}^{x}+f_{2}\left(\epsilon\right)d\sigma_{\epsilon}^{x}$
where
\begin{equation}
f_{1}\left(\epsilon\right)\equiv\frac{\epsilon I_{1}\left(\epsilon\right)-2I_{2}\left(\epsilon\right)}{I_{1}^{2}\left(\epsilon\right)-I_{0}\left(\epsilon\right)I_{2}\left(\epsilon\right)}\mbox{ and }f_{2}\left(\epsilon\right)\equiv\frac{-\epsilon I_{2}\left(\epsilon\right)}{I_{1}^{2}\left(\epsilon\right)-I_{0}\left(\epsilon\right)I_{2}\left(\epsilon\right)},\label{eq:f1 and f2}
\end{equation}
from which one sees that $\mu_{\epsilon}^{x}$ has the ``right''
limit as $\epsilon\downarrow0$. In addition, we can apply more Bessel
function identities:
\[
I_{1}^{\prime}\left(\epsilon\right)=\frac{1}{\epsilon}I_{1}\left(\epsilon\right)+I_{2}\left(\epsilon\right),\; I_{0}\left(\epsilon\right)-I_{2}\left(\epsilon\right)=\frac{2}{\epsilon}I_{1}\left(\epsilon\right),
\]
\[
K_{1}^{\prime}\left(\epsilon\right)=\frac{1}{\epsilon}K_{1}\left(\epsilon\right)-K_{2}\left(\epsilon\right)\mbox{ and }I_{1}\left(\epsilon\right)K_{2}\left(\epsilon\right)+I_{2}\left(\epsilon\right)K_{1}\left(\epsilon\right)=\frac{1}{\epsilon},
\]
to write down $\mathbf{B}^{-1}\left(\epsilon\right)\mathbf{A}\left(\epsilon\right)$
explicitly as $\left(I_{1}^{2}\left(\epsilon\right)-I_{0}\left(\epsilon\right)I_{2}\left(\epsilon\right)\right)^{-1}$
times 
\begin{multline*}
\left(\begin{array}{cc}
-\left(1+\epsilon^{-2}\right) & I_{1}\left(\epsilon\right)K_{1}\left(\epsilon\right)+I_{2}\left(\epsilon\right)K_{0}\left(\epsilon\right)+\epsilon^{-2}\\
I_{1}\left(\epsilon\right)K_{1}\left(\epsilon\right)+I_{2}\left(\epsilon\right)K_{0}\left(\epsilon\right)+\epsilon^{-2} & -\epsilon^{-2}
\end{array}\right).
\end{multline*}
Therefore we have: \\
(1), given $x\in\mathbb{R}^{4}$ and $\epsilon_{1}\geq\epsilon_{2}>0$,
\[
\begin{split}\mathbb{E}^{\mathcal{W}}\left[\mathcal{I}\left(h_{\mu_{\epsilon_{1}}^{x}}\right)\mathcal{I}\left(h_{\mu_{\epsilon_{2}}^{x}}\right)\right]=G\left(\epsilon\right) & \equiv\left(-\frac{1}{4\pi^{2}}\right)\zeta^{\top}\mathbf{B}^{-1}\left(\epsilon\right)\mathbf{A}\left(\epsilon\right)\zeta\;(\mbox{with }\zeta^{\top}=\left(1,1\right))\\
 & =\left(-\frac{1}{4\pi^{2}}\right)\frac{2I_{1}\left(\epsilon\right)K_{1}\left(\epsilon\right)+2I_{2}\left(\epsilon\right)K_{0}\left(\epsilon\right)-1}{I_{1}^{2}\left(\epsilon\right)-I_{0}\left(\epsilon\right)I_{2}\left(\epsilon\right)};
\end{split}
\]
(2), given $x,y\in\mathbb{R}^{4}$, $x\neq y$, and $\epsilon_{1},\epsilon_{2}>0$
with $\epsilon_{1}>\left|x-y\right|+\epsilon_{2}$, $\mathbb{E}^{\mathcal{W}}\left[\mathcal{I}\left(h_{\mu_{\epsilon_{1}}^{x}}\right)\mathcal{I}\left(h_{\mu_{\epsilon_{2}}^{x}}\right)\right]$
equals $\frac{-1}{2\pi^{2}}\left(I_{1}^{2}-I_{0}I_{2}\right)^{-1}\left(\epsilon_{1}\right)$
times
\[
\begin{split} & \zeta^{\top}\mathbf{B}^{-1}\left(\epsilon_{1}\right)\mathbf{A}\left(\epsilon_{1}\right)\mathbf{C}\left(\left|x-y\right|\right)\zeta.\\
= & \left(\begin{array}{cc}
\left(I_{1}K_{1}+I_{2}K_{0}\right)\left(\epsilon_{1}\right)-1 & \left(I_{1}K_{1}+I_{2}K_{0}\right)\left(\epsilon_{1}\right)\end{array}\right)\left(\begin{array}{c}
\frac{I_{1}\left(\left|x-y\right|\right)}{\left|x-y\right|}\\
I_{2}\left(\left|x-y\right|\right)+\frac{I_{1}\left(\left|x-y\right|\right)}{\left|x-y\right|}
\end{array}\right)\\
= & \left(\frac{2I_{1}\left(\left|x-y\right|\right)}{\left|x-y\right|}+I_{2}\left(\left|x-y\right|\right)\right)\left(\left(I_{1}K_{1}+I_{2}K_{0}\right)\left(\epsilon_{1}\right)-\frac{1}{2}\right)+\frac{1}{2}I_{2}\left(\left|x-y\right|\right)\\
= & I_{0}\left(\left|x-y\right|\right)\left(\left(I_{1}K_{1}+I_{2}K_{0}\right)\left(\epsilon_{1}\right)-\frac{1}{2}\right)+\frac{1}{2}I_{2}\left(\left|x-y\right|\right);
\end{split}
\]
(3), given $x,y\in\mathbb{R}^{4}$, $x\neq y$, and $\epsilon_{1},\epsilon_{2}>0$
with $\left|x-y\right|>\epsilon_{1}+\epsilon_{2}$,
\[
\begin{split}\mathbb{E}^{\mathcal{W}}\left[\mathcal{I}\left(h_{\mu_{\epsilon_{1}}^{x}}\right)\mathcal{I}\left(h_{\mu_{\epsilon_{2}}^{x}}\right)\right] & =\left(-\frac{1}{2\pi^{2}}\right)\zeta^{\top}\mathbf{D}\left(\left|x-y\right|\right)\zeta.\\
 & =\left(-\frac{1}{2\pi^{2}}\right)\left(\frac{2K_{1}\left(\left|x-y\right|\right)}{\left|x-y\right|}-K_{2}\left(\left|x-y\right|\right)\right)\\
 & =\frac{1}{2\pi^{2}}K_{0}\left(\left|x-y\right|\right).
\end{split}
\]

\end{document}